\documentclass{article}

\usepackage[english]{babel}

\usepackage{amssymb,amsmath,amsfonts,amssymb,amsthm}
\usepackage{graphics,graphicx,color,hyperref}

\usepackage{subfig}
\usepackage{xcolor}
\usepackage[font=small]{caption}

\providecommand{\norm}[1]{\lVert#1\rVert}
\theoremstyle{plain}
\newtheorem{theorem}{Theorem}
\newtheorem{lemma}{Lemma}
\newtheorem{proposition}{Proposition}

\newtheorem*{mainthm}{Theorem \ref{thm:main}}
\theoremstyle{definition}
\newtheorem{definition}{Definition}
\theoremstyle{remark}
\newtheorem{remark}{Remark}

\def \Rm {\mathbb R}

\newcommand{\eps}{\varepsilon}
\renewcommand{\epsilon}{\varepsilon}

\newcommand{\dsum}{\displaystyle\sum}

\newcommand{\pdrr}[2]{\dfrac{\partial^2{#1}}{\partial{#2}^2}}

\newcommand{\vp}{\varphi}

\newcommand{\cout}[1]{}

\newcommand{\Xu}{X^1}
\newcommand{\Xd}{X^2}
\newcommand{\Xii}{X^i}

\title{Critical Points for Elliptic Equations with Prescribed Boundary Conditions\thanks{G.\ S.\ Alberti acknowledges support from the ETH Z\"urich Postdoctoral Fellowship Program as well as from the Marie Curie Actions for People COFUND Program. G. Bal acknowledges partial support from the National Science Foundation and from the Office of Naval Research. }}

\author{Giovanni S.\ Alberti \thanks{Dipartimento di Matematica, Universit\`a di Genova, Via Dodecaneso 35, 16146 Genova, Italy. Email: alberti@dima.unige.it.}
        \and Guillaume Bal \thanks{Department of Applied Physics and
        Applied Mathematics, Columbia University,
        New York NY, 10027, USA. Email: gb2030@columbia.edu.}
        \and Michele Di Cristo \thanks{Dipartimento di Matematica, Politecnico di Milano, Piazza Leonardo da Vinci 32, 20133 Milano, Italy. Email: michele.dicristo@polimi.it.}}
        
\date{February 3, 2017}

\begin{document}

\maketitle

\begin{abstract}This paper concerns the existence of critical points for solutions to second order elliptic equations of the form $\nabla\cdot \sigma(x)\nabla u=0$ posed on a bounded domain $X$ with {\em prescribed} boundary conditions. In spatial dimension $n=2$, it is known that the number of critical points (where $\nabla u=0$) is related to the number of oscillations of the boundary condition  independently of the (positive) coefficient $\sigma$. We show that the situation is  different in dimension $n\geq3$. More precisely, we obtain that for any fixed (Dirichlet or Neumann) boundary condition for $u$ on $\partial X$, there exists an open set of smooth coefficients $\sigma(x)$ such that $\nabla u$ vanishes at least at one point in $X$.
By using estimates related to the Laplacian with mixed boundary conditions, the result is first obtained for a piecewise constant conductivity with infinite contrast, a problem of independent interest. A second step shows that the topology of the vector field $\nabla u$ on a subdomain is not modified for appropriate bounded, sufficiently high-contrast, smooth coefficients $\sigma(x)$.

These results find applications in the class of hybrid inverse problems, where optimal stability estimates for parameter reconstruction are obtained in the absence of critical points. Our results show that for any (finite number of) prescribed boundary conditions, there are coefficients $\sigma(x)$ for which the stability of the reconstructions will inevitably degrade.
\end{abstract}

\vspace{2mm}

\section{Introduction}

Consider a bounded Lipschitz domain $X\subset\Rm^n$ and a prescribed boundary condition $g\in C^0(\partial X)\cap H^{\frac{1}{2}}(\partial X)$.
 We want to assess the existence of coefficients $\sigma(x)$ (referred to as {\em conductivities}) so that the solution $u$ of the following elliptic problem
\begin{equation}
\label{eq:ell}
 -\nabla\cdot\sigma\nabla u =0 \quad \mbox{ in } X,\qquad u=g \quad\mbox{ on } \partial X
\end{equation}
admits at least one critical point $x\in X$, i.e.\  $\nabla u(x)=0$.

The analysis of this problem is markedly different in dimension $n=2$ and dimensions $n\geq3$. In the former case, it is indeed known that critical points are isolated and their number is given by the number of oscillations of $g$ minus one, independently of the coefficient $\sigma(x)$ (bounded above and below by positive constants and of class $C^{0,\alpha}$); see \cite{alessandrinimagnanini1994,A-AMPA-86}. This no longer holds in dimension $n\geq3$, where the set of critical points can be quite complicated \cite{CF-JDE-85,HaHoHoNa-JDG99}. However, as far as the authors are aware, it has not been known whether it is possible to construct boundary values independently of $\sigma$ so that the corresponding solutions do not have critical points. The main contribution of this paper is a negative answer to this question.
\begin{theorem}\label{thm:main}
Let $X\subset\Rm^3$ be a bounded Lipschitz domain. Take $g\in C(\partial X)\cap H^{\frac12}(\partial X)$. Then there exists a nonempty open set of conductivities $\sigma\in C^\infty(\overline{X})$, $\sigma\ge 1/2$, such that the solution $u\in H^1(X)$ to
\begin{equation*}
 -\nabla\cdot\sigma\nabla u =0 \quad \mbox{ in } X,\qquad u=g \quad\mbox{ on } \partial X
\end{equation*}
has a critical point in $X$, namely $\nabla u(x)=0$ for some $x\in X$ (depending on $\sigma$).
\end{theorem}
We consider the case $n=3$ for concreteness of notation, but our results may be easily generalized to the case $n\ge 3$. The above result may be extended to the case of an arbitrary finite number of boundary conditions (see Theorem~\ref{thm:multiple} for the precise statement), to the case of an arbitrary finite number of critical points located in arbitrarily small balls given a priori (Theorem~\ref{thm:main-local}), as well as to the case of Neumann boundary conditions (Theorem~\ref{thm:main-neumann}).

The main idea of the construction is similar to the use of interlocked rings to show that the determinant of $n$ gradients $\nabla u_i$ may change sign in dimension $n\geq3$ \cite{BMN-ARMA-04} (see also \cite{ancona-2002} for the case of critical points), a result that cannot hold in dimension $n=2$ \cite{AN-ARMA-01,bauman2001univalent}. More precisely, let $x_0$ be a point in $X$ and $S$ the surface of a subdomain $Z\subset X$ enclosing $x_0$. We separate $S$ into two disjoint subsets $S_1\cup S_2$ such that the harmonic solution in $Z$ equal to $i$ on $S_i$  has a critical point in $x_0$; see for instance Fig.~\ref{fig:domain} where $S_1$ is the ``circular" part of the boundary of a cylinder while $S_2$ is the ``flat" part of that boundary. Note that at least one of the domains $S_i$ is not connected. Consider the case when $g$ takes at least two values, say, $1$ and $2$ after proper rescaling. For $i=1,2$, let now $\Xii$  be two handles (open domains) joining $S_i$ to points $x_{(i)}$ on $\partial X$ where $g(x_{(i)})=i$. For appropriate choices of $S_i$, the handles $\Xii$ may be shown not to intersect in dimension $n\geq3$, whereas they clearly have to intersect in dimension $n=2$. 
Let us now assume that $\sigma$ is set to $+\infty$ in both handles and equal to $1$ otherwise. This forces the solution $u$ to equal $i$ on $S_i$, to be harmonic in $Z$, and hence to have a critical point in $x_0$. It remains to show that the topology of the vector field $\nabla u$ is not modified in the vicinity of $x_0$ when $\sigma$ is replaced by a sufficiently high-contrast (and possibly smooth) conductivity. 
This proves the existence of critical points for arbitrarily prescribed Dirichlet conditions for some open set of conductivities.

Let us conclude this introductory section by mentioning applications of the aforementioned results to hybrid inverse problems. The latter class of problems typically involves a two step inversion procedure. The first step  provides volumetric information about unknown coefficients of interest. The simplest example of such information is the solution $u$ itself in a problem of the form $\nabla\cdot \sigma(x)\nabla u=0$. The second step of the procedure then aims to reconstruct the unknown coefficients from such information; in the considered example, the conductivity $\sigma(x)$. We refer the reader to \cite{alberti-capdeboscq-2016,A-Sp-08,ABCTF-SIAP-08,ammari-2013,AS-IP-12,B-IO-12,B-CM-14,BR-IP-12,BU-IP-10,BU-CPAM-13,CFGK-SJIS-09,KS-IP-12,MR3289684,MZM-IP-10,NTT-Rev-11,S-SP-2011,widlak-scherzer-2012} and their references for additional information on these inverse problems.

It should be clear from the above example that the reconstruction of $\sigma$ is better behaved when $\nabla u$ does not vanish. In the aforementioned works, results of the following form have been obtained: for each reasonable conductivity $\sigma$, there is an open set of, say, Dirichlet boundary conditions such that $|\nabla u|$ is bounded from below by a positive constant. What our results show is that in dimension $n\geq3$, there is no {\em universal} finite set of Dirichlet boundary conditions for which $|\nabla u|$ is bounded from below by a positive constant uniformly in $\sigma$, which is the condition guaranteeing optimal stability estimates with respect to measurement noise. In other words, {\em optimal}  (in terms of stability) boundary conditions, which may be designed by the practitioner, depend on the (unknown) object we wish to reconstruct; see, e.g., \cite{BC-JDE-13} for such a possible construction. For Helhmoltz-type problems, suitable boundary conditions may be constructed a priori, i.e.\ independently of the parameters, at the price of taking measurements at several frequencies \cite{alberti-2013,alberti-2015b,alberti-2015,alberti-2016,capalb-analytic}.

Note that other, practically less optimal, stability results may be obtained even in the presence of critical points \cite{AL-DC-FR-VE} or nodal points \cite{alessandrini-2014}.  Also, the presence of critical points is not the only qualitative feature of interest in hybrid inverse problems. A result similar to ours in the setting of the sign of the determinant of solution gradients has been recently obtained in \cite{alberti-capdeboscq-2016,CA}. However, this method does not immediately extend to the case of critical points.

This paper is structured as follows. Our main results on the existence of critical points for well-chosen conductivities are presented in section~\ref{sec:main}, first for Dirichlet boundary conditions in $\S$\ref{sec:mainD} and then for Neumann boundary conditions in $\S$\ref{sec:mainN}.
The proofs of these theorems are based on some auxiliary results, which are presented in the rest of the paper. In section~\ref{sec:zaremba} we discuss the Zaremba problem, which concerns the analysis of harmonic functions with mixed boundary values. Finally, in section~\ref{sec:highc} we generalize the high-contrast results of \cite{caloz2010uniform} to the case of inclusions touching the boundary (to address the case of the aforementioned handles). The latter result, obtained for Dirichlet boundary conditions in $\S$\ref{sub:dir}, is modified in $\S$\ref{sec:Nasymp} to treat the case of Neumann  conditions.

\section{Existence of Critical Points}
\label{sec:main}

We now construct a geometry that guarantees the existence of critical points in the infinite contrast setting. We then argue by continuity to obtain the existence of critical points for finite but large contrasts. We first consider the setting with prescribed Dirichlet boundary conditions.

The proofs of this section  make use of the auxiliary results contained in sections~\ref{sec:zaremba} and \ref{sec:highc}.

\subsection{Dirichlet Boundary Conditions}
\label{sec:mainD}
We first state the following technical lemma that allows us to control the harmonic solutions in the handles $\Xii$ in the infinite contrast setting.
\begin{lemma}\label{lem:family}
Let $X\subset\Rm^3$ be a bounded Lipschitz domain. Take $x_0\in\partial X$ and $g\in C(\partial X)\cap H^{\frac12}(\partial X)$. For $\rho\in (0,1)$ consider a family of  subdomains $X_\rho\subset X$  such that
\begin{enumerate}
\item $\partial X_\rho\cap\partial X=B(x_0,\rho)\cap\partial X$;
\item and $X_\rho$ are uniformly Lipschitz (according to \cite[Definition 12.10]{leoni}), with constants independent of $\rho$.
\end{enumerate}
Let $u_\rho\in H^1(X_\rho)$ be the solution of
\[
 \left\{ \begin{array}{l}
         -\Delta u_\rho = 0\quad\text{in $X_\rho$,}\\
         u_\rho=g\quad\text{on $\partial X_\rho\cap\partial X$,}\\
        \partial_\nu u_\rho=0  \quad\text{on $\partial X_\rho\setminus\partial X$.}
        \end{array}
        \right.
\]
Then
\[
\lim_{\rho\to 0}\, \norm{u_\rho - g(x_0)}_{H^{\frac12}(\partial X_\rho)}=0.
\]
\end{lemma}
\begin{proof}
We denote  several positive constants independent of $\rho$ and $g$ by $C$. Set $D_\rho = ( \partial X_\rho\cap\partial X)^\circ$ and $N_\rho = \partial X_\rho\setminus\partial X$. We first note that, by assumption 2, the trace operator in $X_\rho$ is uniformly bounded, namely
\begin{equation}\label{eq:uniform_trace}
\norm{u}_{H^{\frac12}(\partial X_\rho)} \le C \norm{u}_{H^1(X_\rho)},\qquad u\in H^1(X_\rho),
\end{equation}
see \cite[Exercise 15.25]{leoni}. Similarly, thanks to assumption 1, by \cite{rychkov1999}, we have that the extension operator ${\rm Ext}_{D_\rho}\colon H^{\frac12}(D_\rho)\to H^{\frac12}(\partial X)$ given by Lemma~\ref{lem:traces}, part 3, is uniformly bounded, namely:
\begin{equation}\label{eq:uniform_trace_2}
\norm{{\rm Ext}_{D_\rho}}_{H^{\frac12}(D_\rho)\to H^{\frac12}(\partial X)} \le C.
\end{equation}

 The difference $v_\rho = u_\rho - g(x_0)\in H^1(X_\rho)$ solves
\[
 \left\{ \begin{array}{ll}
         -\Delta v_\rho = 0&\quad\text{in $X_\rho$,}\\
         v_\rho=g-g(x_0)&\quad\text{on $D_\rho$,}\\
        \partial_\nu v_\rho=0  &\quad\text{on $N_\rho$.}
        \end{array}
        \right.
\]
Integrating by parts yields
\[
\norm{\nabla v_\rho}_{L^2(X_\rho)}^2=\int_{X_\rho} |\nabla v_\rho|^2\,dx = \int_{\partial X_\rho} v_\rho \, \partial_\nu v_\rho\,ds=\int_{D_\rho} (g-g(x_0)) \, \partial_\nu v_\rho\,ds.
\]
Set $w_\rho={\rm Ext}_{\partial X} {\rm Ext}_{D_\rho}(g-g(x_0))\in H^1(X)$. Integrations by parts give
\[
\norm{\nabla v_\rho}_{L^2(X_\rho)}^2=\int_{\partial X_\rho} w_\rho \, \partial_\nu v_\rho\,ds = \int_{X_\rho} \nabla v_\rho\cdot \nabla w_\rho\,dx\le \norm{\nabla w_\rho}_{L^2(X_\rho)}\norm{\nabla v_\rho}_{L^2(X_\rho)},
\]
which yields
\[
\begin{split}
\norm{\nabla v_\rho}_{L^2(X_\rho)}  &\le \norm{\nabla w_\rho}_{L^2(X)} \\
& \le
\norm{{\rm Ext}_{\partial X}}\norm{ {\rm Ext}_{D_\rho}}\norm{g-g(x_0)}_{H^{\frac12}(D_\rho)}\\&\le C\norm{g-g(x_0)}_{H^{\frac12}(D_\rho)},
\end{split}
\]
where the last inequality follows from \eqref{eq:uniform_trace_2}.
Moreover, the Hopf lemma yields
\[
\norm{v_\rho}_{L^2(X_\rho)}\le C\norm{v_\rho}_{L^\infty(X_\rho)}\le C\norm{g-g(x_0)}_{L^\infty(D_\rho)}.
\]

Combining these two inequalities we obtain
\[
\norm{v_\rho}_{H^1(X_\rho)}\le C\bigl(\norm{g-g(x_0)}_{H^{\frac12}(D_\rho)}+\norm{g-g(x_0)}_{L^\infty(D_\rho)}\bigr).
\]
As a consequence, by \eqref{eq:uniform_trace} we have
\[
\norm{v_\rho}_{H^{1/2}(\partial X_\rho)}\le C\bigl(\norm{g-g(x_0)}_{H^{\frac12}(D_\rho)}+\norm{g-g(x_0)}_{L^\infty(D_\rho)}\bigr).
\]
Finally, by continuity of $g$ and assumption 1, $\norm{ g-g(x_0)}_{L^{\infty}(\partial X_\rho\cap\partial X)}\to 0$ as $\rho\to 0$. Moreover, by the fact that $g\in H^{\frac12}(\partial X)$ and assumption 1, $ \norm{ g-g(x_0)}_{H^{\frac12}(\partial X_\rho\cap\partial X)}\to 0$  as $\rho\to 0$. This concludes the proof.
\end{proof}

We are now ready to prove our main result.
\begin{mainthm}\label{thm:main2}
Let $X\subset\Rm^3$ be a bounded Lipschitz domain. Take $g\in C(\partial X)\cap H^{\frac12}(\partial X)$. Then there exists a nonempty open set of conductivities $\sigma\in C^\infty(\overline{X})$, $\sigma\ge 1/2$, such that the solution $u\in H^1(X)$ to
\begin{equation*}
 -\nabla\cdot\sigma\nabla u =0 \quad \mbox{ in } X,\qquad u=g \quad\mbox{ on } \partial X
\end{equation*}
has a critical point in $X$, namely $\nabla u(x)=0$ for some $x\in X$ (depending on $\sigma$).
\end{mainthm}
\begin{remark}
Note that such {\em pathological} conductivities $\sigma$ will necessarily have sufficiently high contrast. Indeed, take for example $g(x)=x_1$: if $\sigma$ is sufficiently close to $\sigma_0\equiv 1$ in the $C^{0,\alpha}$ norm, then standard Schauder estimates yield that $\nabla u \approx (1,0,0)$ uniformly, and so critical points do not exist.
\end{remark}

\begin{proof}
If $g$ is constant, then the result is obvious. Thus, assume that there exist  $x_{(1)},x_{(2)}\in\partial X$ such that $g(x_{(1)})\neq g(x_{(2)})$. Without loss of generality, we assume that $g(x_{(i)})=i$ for $i=1,2$. Let us precisely discuss how to construct the subdomains where the conductivity will have very large values. These subdomains  will depend on a small parameter $\rho\in (0,\tilde \rho)$ to be fixed later.
\begin{figure}
\begin{center}
\includegraphics[width=.89\textwidth]{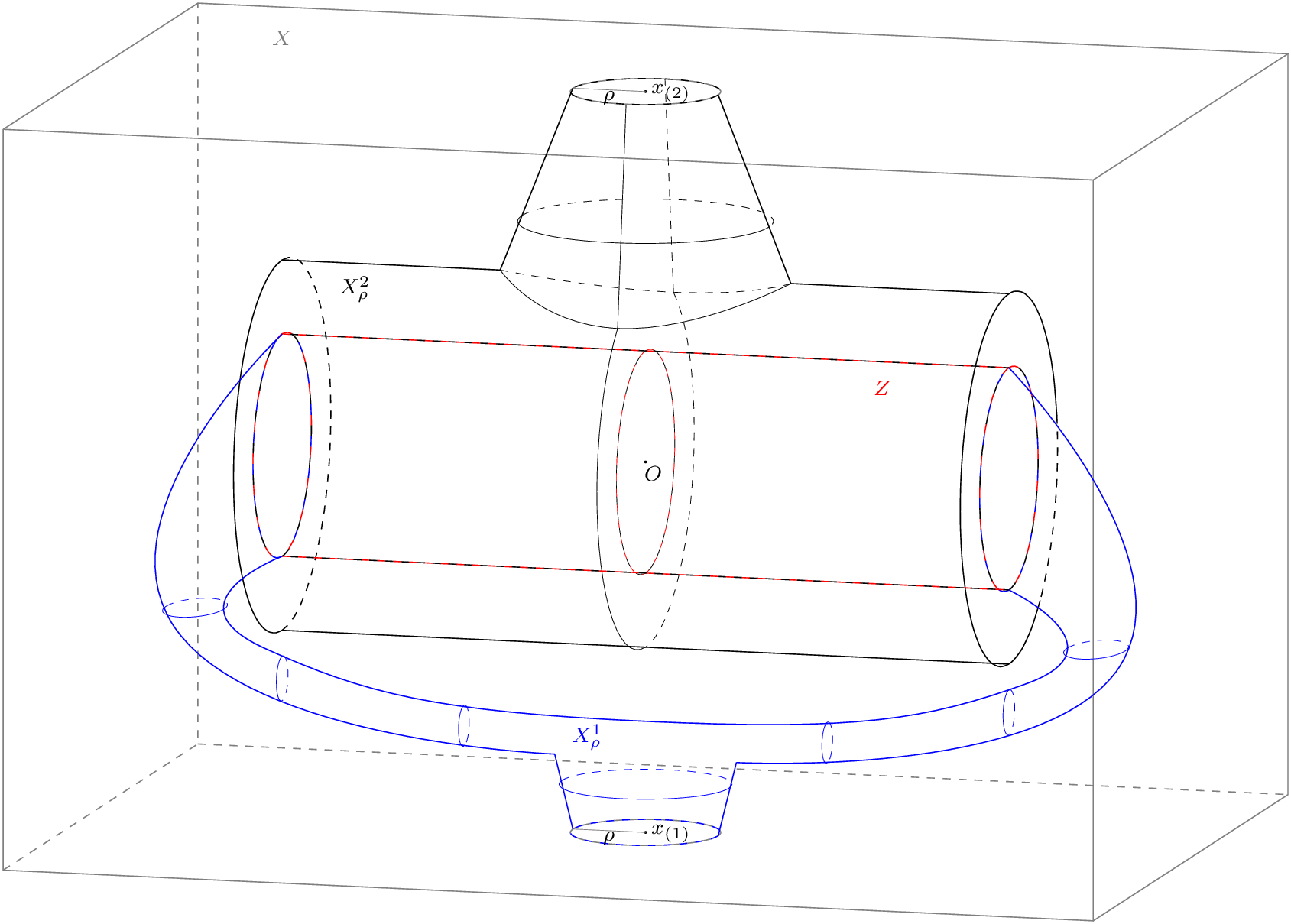}
\caption{The subdomains $Z$ and $X^i_{\rho}$. \label{fig:domain}}
\end{center}
\end{figure}

\medskip

\emph{Step 1: Construction of the subdomains.}
See Figure~\ref{fig:domain}.
Let $Z$ be the cylinder given by $Z=\{x\in\mathbb R^3:x_2^2+x_3^2<1, |x_1|<2\}$.
Without loss of generality, we assume that $X$ is connected and that $\overline{Z}\subset X$. The two lateral discs of the cylinder $Z$ are connected to $x_{(1)}$ with a Lipschitz subdomain $X^1_{\rho}$ satisfying the assumptions of Lemma~\ref{lem:family}. Similarly, the lateral surface of $Z$ is connected to $x_{(2)}$ with a Lipschitz subdomain $X^2_{\rho}$ satisfying the assumptions of Lemma~\ref{lem:family}. In particular, $\partial X^i_{\rho}\cap\partial X=B(x_{(i)},\rho)\cap\partial X$ for $i=1,2$. Moreover, we choose $X^i_{\rho}$ in such a way that $X^i_{\rho}$, with respect to the decomposition of the boundary given by $(\partial X^i_{\rho}\cap\partial X)^\circ$ and $\partial X^i_{\rho}\setminus\partial X$, is creased, according to Definition~\ref{def:creased}. In essence, this means that  $(\partial X^i_{\rho}\cap\partial X)^\circ$ and $\partial X^i_{\rho}\setminus\partial X$ are separated by a Lipschitz interface and that the angle between them is smaller than $\pi$.

\medskip

\begin{figure}[t]

\begin{centering}
\subfloat[The field $\nabla u^*$]{\label{fig:multi-frequency_1Da}\includegraphics[width=.42\textwidth]{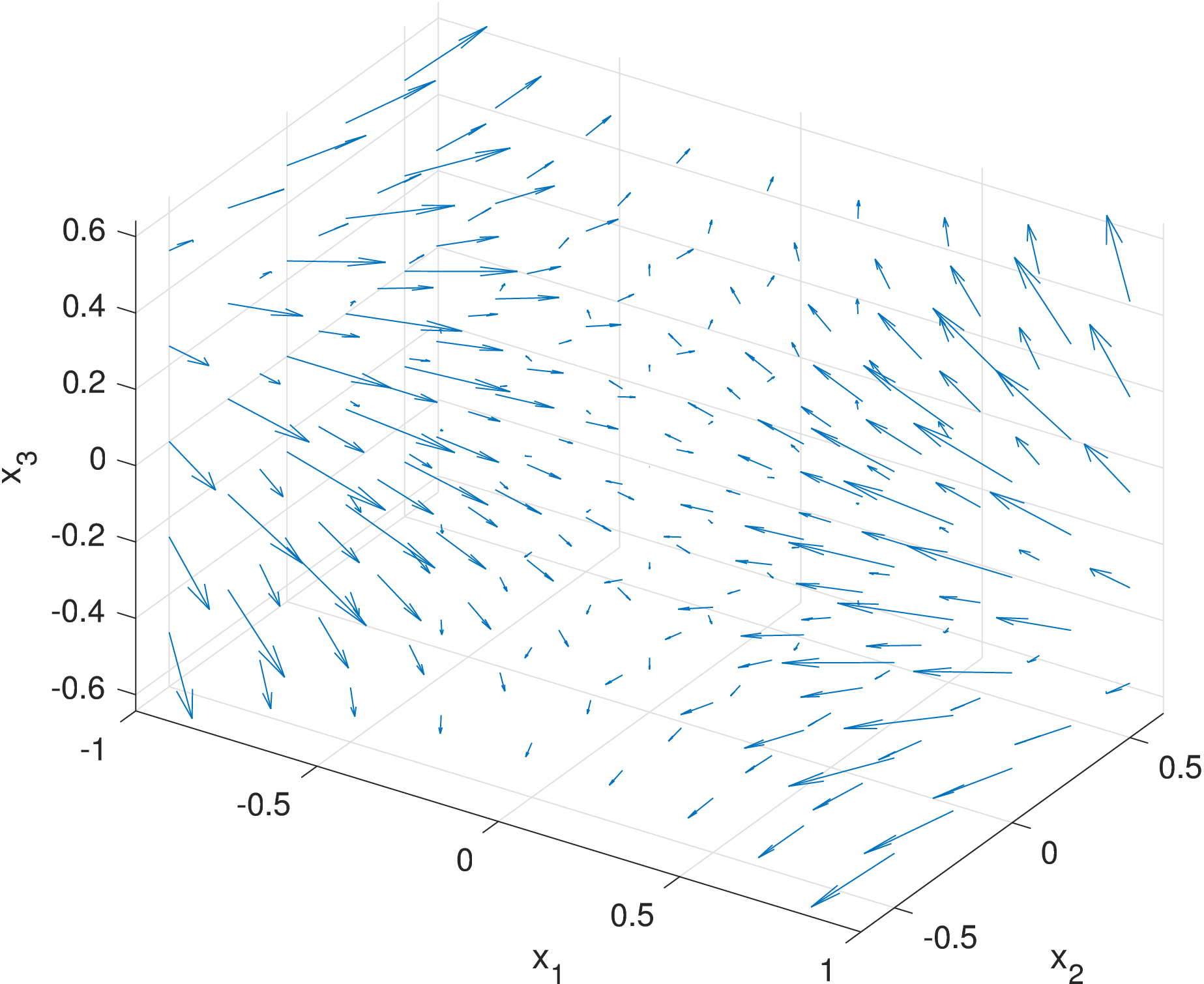}
}
\qquad\subfloat[The field $\nabla u^*$ on the section $x_2=0$]{\label{fig:multi-frequency_1Db}\includegraphics[width=.39\textwidth]{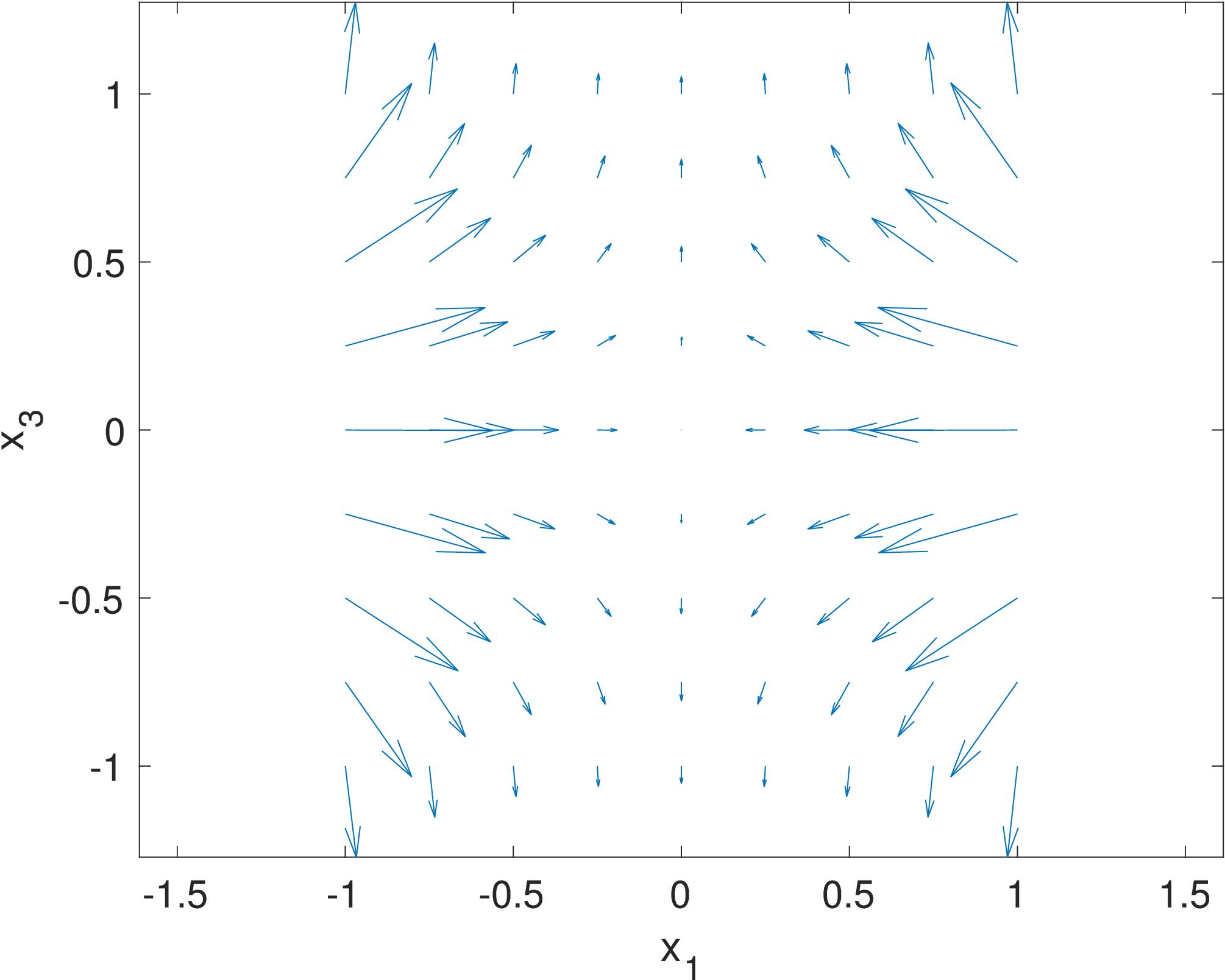}}

\subfloat[The field $R\nabla u^*$]{\label{fig:multi-frequency_1Da}\includegraphics[width=.42\textwidth]{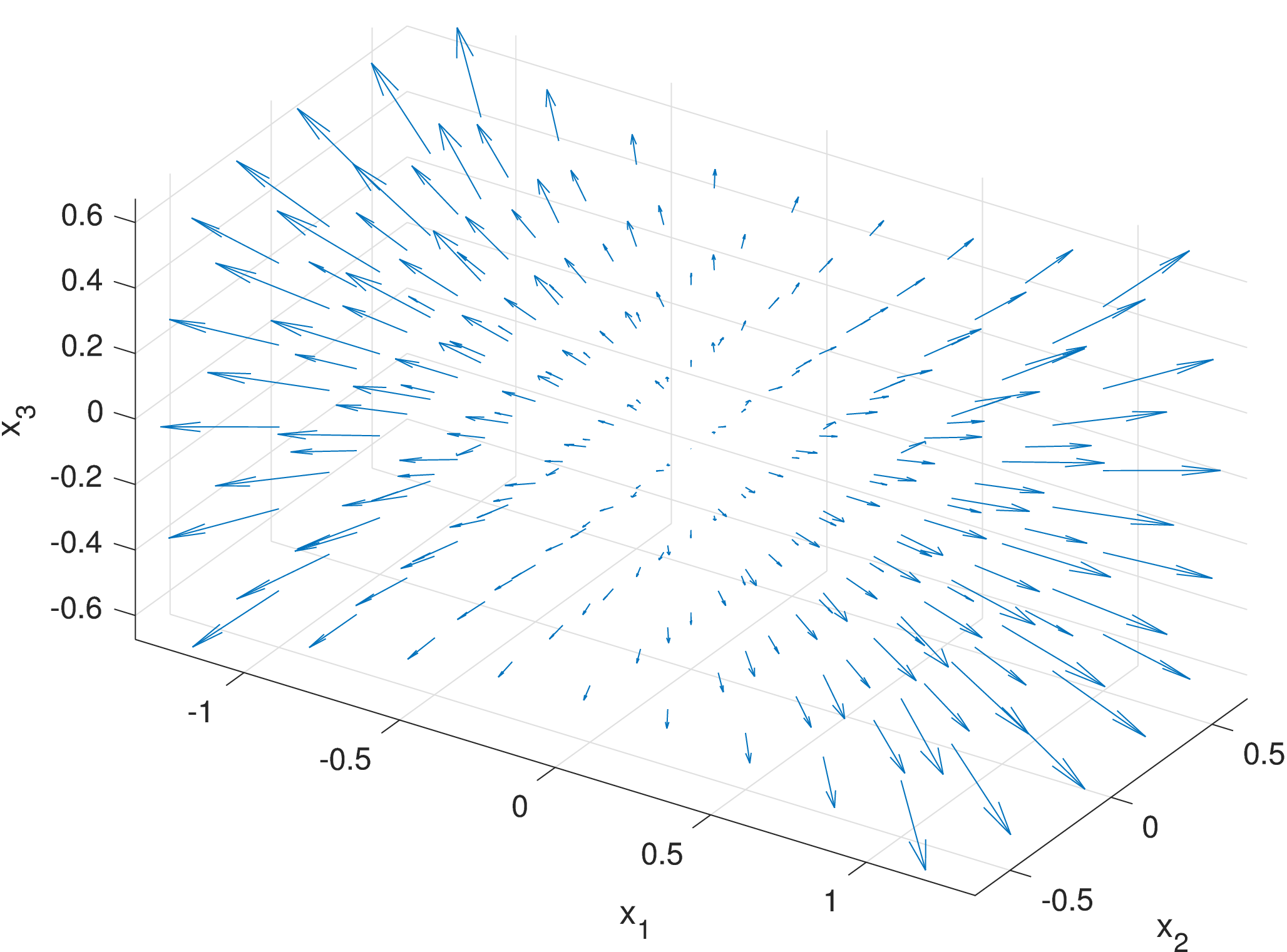}
}
\qquad\subfloat[The field $R\nabla u^*$ on the section $x_2=0$]{\label{fig:multi-frequency_1Db}\includegraphics[width=.39\textwidth]{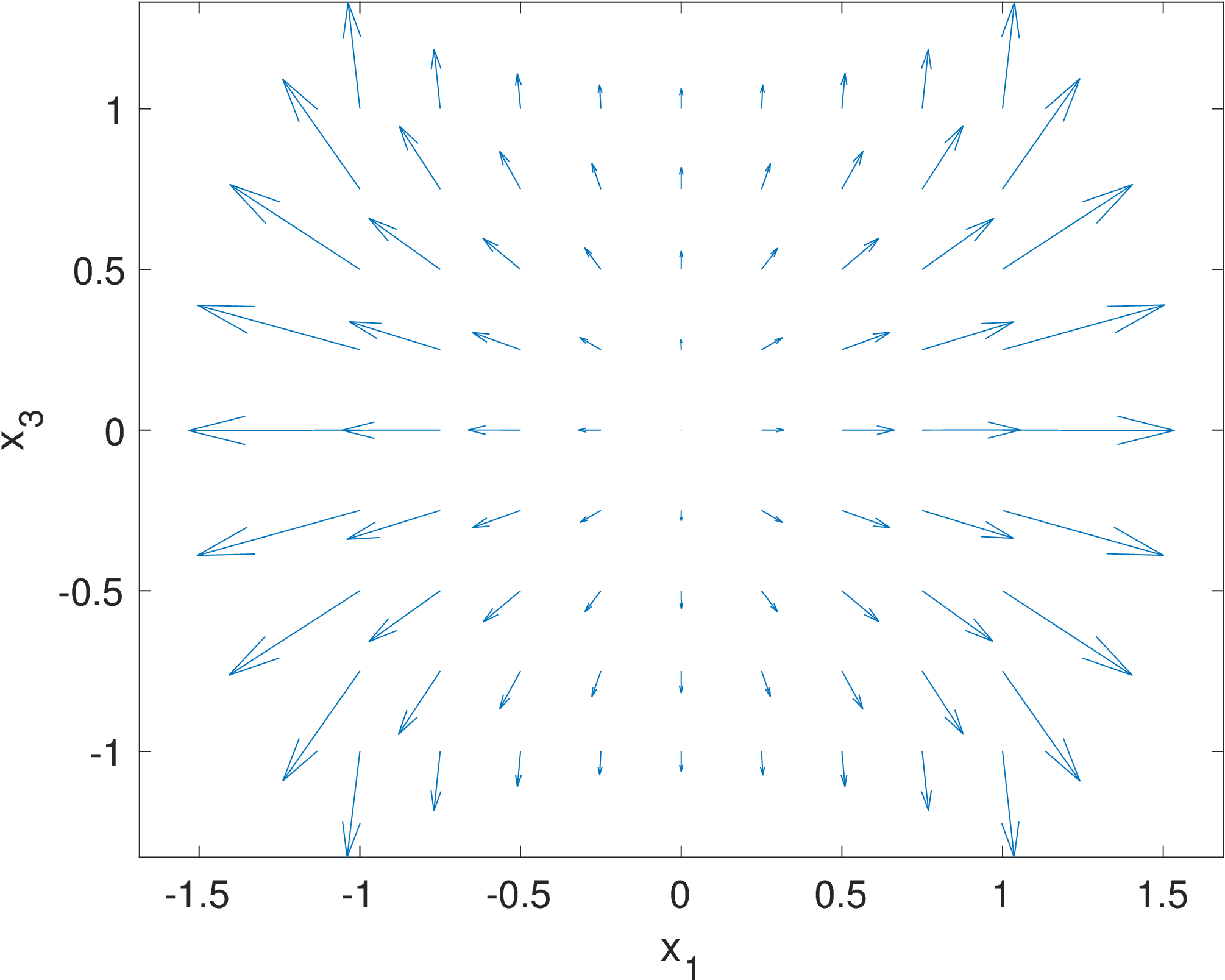}
}

\caption{\label{fig:gradients}The fields $\nabla u^*$ and $R\nabla u^*$ near the origin.}
\end{centering}
\end{figure}

\emph{Step 2: The limiting case in $Z$ as $\eta\to 0$ and $\rho\to 0$.} Let $u^*\in H^{\frac34}(Z)$ be the unique weak solution (existence and uniqueness follow from Lemma~\ref{lem:jumps} and Proposition~\ref{prop:zaremba}) to
\begin{equation}\label{eq:ustar-pde}
 \left\{ \begin{array}{ll}
         -\Delta u^* = 0&\quad\text{in $Z$,}\\
         u^*=1&\quad\text{on $\partial Z\cap\partial X^1_{\rho}$,}\\
         u^*=2&\quad\text{on $\partial Z\cap\partial X^2_{\rho}$.}
        \end{array}
        \right.
\end{equation}
By the symmetries of the domain $Z$ and of the boundary values of $u^*$, we have that $u^*$ is even with respect to $x_1$ and radially symmetric with respect to $(x_2,x_3)$. Therefore, setting $O=(0,0,0)$, we have
\[
\nabla u^*(O) = 0,\qquad  \partial_{x_i x_j} u^*(O) = 0,\;i\neq j.
\]
As a consequence, since $u^*$ is harmonic, the Hessian of $u^*$ at $O$ is of the form ${\rm Diag(-2\lambda,\lambda,\lambda)}$ for some $\lambda\in\mathbb{R}$. We now show that $\lambda>0$ (see Figure~\ref{fig:gradients}).
Consider the function $u^*_s$ on $Z_+ = \{x\in\mathbb{R}^3:0<x_1<4,x_2^2+x_3^2<1\}$ defined by
\[
u^*_s(x_1,x_2,x_3)=
\begin{cases}
u^*(x_1,x_2,x_3) &\text{if $x_1\le 2$,}\\
2-u^*(4-x_1,x_2,x_3) &\text{if $x_1> 2$.}
\end{cases}
\]
By construction, since $u^*_s$ and $\partial_{x_1} u^*_s$ are continuous across $\{x_1=2\}$, we have that $u^*_s$ is harmonic in $Z_+$. Thus, the function $v=\partial_{x_1} u^*_s$ is harmonic in $Z_+$ as well. Since $u^*$ is even with respect to $x_1$, we have $v=0$ on $\partial Z_+ \cap \{x_2^2+x_3^2<1\}$. Moreover, since $u^*_s=2$ on $\partial Z_+ \cap \{0<x_1<2\}$ and $u^*_s=0$ on $\partial Z_+ \cap \{2<x_1<4\}$, we have $v = -2\delta_{\{x_1=2\}}$ on $\partial Z_+ \cap \{0<x_1<4\}$. We have proven that
\begin{equation*}
 \left\{ \begin{array}{ll}
         -\Delta v = 0&\quad\text{in $Z_+$,}\\
         v=0&\quad\text{on $\partial Z_+\cap\{x_2^2+x_3^2<1\}$,}\\
         v=-2\delta_{\{x_1=2\}}&\quad\text{on $\partial Z_+\cap\{x_2^2+x_3^2=1\}$.}
        \end{array}
        \right.
\end{equation*}
Thus, by the maximum principle we obtain that $v\le 0$ in $Z_+$. Finally,  the Hopf lemma applied to $v_{|\{0<x_1<1\}}$ yields that
 $\partial^2_{x_1} u^*(O)=\partial_{x_1} v(O)<0$, namely $\lambda>0$. The above qualitative argument, which is sufficient for our proof, may be made quantitative by writing an explicit expression for $u^*$ as a series expansion; the reader is referred to appendix~\ref{sec:appendix} for the details.

We have shown that $u^*$ has a saddle point in $O$; more precisely, we have $\nabla u^*(O) = 0$ and $D^2u^*(O) = {\rm Diag(-2\lambda,\lambda,\lambda)}$ with $\lambda>0$.
This implies
\begin{equation}\label{eq:ustar}
\nu\cdot(R\nabla u^*)\ge 8\mu\quad\text{on $\partial B(0,r)$}
\end{equation}
for some $\mu>0$ and $r\in(0,1)$, where $R$ is the diagonal matrix given by $R={\rm Diag}(-1,1,1)$.

\medskip

\emph{Step 3: The limiting case  as $\eta\to 0$ for $\rho$ small enough.} Let $u^{i}_{\rho}\in H^{1}(X^i_{\rho})$ be the unique weak solution (existence and uniqueness follow from  Proposition~\ref{prop:zaremba}) to
\[
 \left\{ \begin{array}{ll}
         -\Delta u^{i}_{\rho} = 0&\quad\text{in $X^i_{\rho}$,}\\
         u^{i}_{\rho}=g&\quad\text{on $\partial X^i_{\rho}\cap\partial X$,}\\
         \partial_\nu u^{i}_{\rho}=0&\quad\text{on $\partial X^i_{\rho}\setminus\partial X$.}
        \end{array}
        \right.
\]
Since $g(x_{(i)})=i$, by Lemma~\ref{lem:family}, we have that
\begin{equation}\label{eq:step2}
\lim_{\rho\to 0}\, \norm{u^{i}_{\rho}-i}_{H^{\frac12}(\partial X^i_{\rho})}=0,\qquad i=1,2.
\end{equation}

Let $u^Z_\rho\in H^{\frac34}(Z)$ be defined by
\[
 \left\{ \begin{array}{ll}
         -\Delta u^Z_\rho = 0&\quad\text{in $Z$,}\\
         u^Z_\rho=u^{1}_{\rho}&\quad\text{on $\partial Z\cap\partial X^1_{\rho}$,}\\
         u^Z_\rho=u^{2}_{\rho}&\quad\text{on $\partial Z\cap\partial X^2_{\rho}$.}
        \end{array}
        \right.
\]
By Lemma~\ref{lem:jumps} and Proposition~\ref{prop:zaremba} we have that
\[
\norm{u^Z_\rho-u^*}_{H^{\frac34}(Z)}\le C (\norm{u^{1}_{\rho}-1}_{H^{\frac12}(\partial Z\cap \partial X^1_{\rho})}+\norm{u^{2}_{\rho}-2}_{H^{\frac12}(\partial Z\cap \partial X^2_{\rho})})
\]
for an absolute constant $C>0$. Therefore, elliptic regularity theory yields
\[
\norm{u^Z_\rho-u^*}_{C^{1}\bigl(\overline{B(0,r)}\bigr)}\le C' (\norm{u^{1}_{\rho}-1}_{H^{\frac12}(\partial Z\cap \partial X^1_{\rho})}+\norm{u^{2}_{\rho}-2}_{H^{\frac12}(\partial Z\cap \partial X^2_{\rho})})
\]
for some $C'>0$ independent of $\rho$, and so by \eqref{eq:step2} we obtain
\begin{equation*}
\lim_{\rho\to 0}\, \norm{u^Z_\rho-u^*}_{C^{1}\left(\overline{B(0,r)}\right)}=0.
\end{equation*}
As a consequence, in view of \eqref{eq:ustar} we can choose $\rho_0>0$ such that
\begin{equation}\label{eq:step2b}
\nu\cdot(R\nabla u^Z_{\rho_0})\ge 4\mu\quad\text{on $\partial B(0,r)$.}
\end{equation}

\medskip

\emph{Step 4: Case with $\rho$ and $\eta$ small enough.} For $\eta\in (0,1)$, define $\sigma_\eta\in L^\infty (X)$ by
\[
 \sigma_\eta=\begin{cases}
              \eta^{-1} & \text{in $X^1_{\rho_0}\cup X^2_{\rho_0}$,} \\
              1 & \text{otherwise.}
             \end{cases}
\]
Let $u_\eta\in H^1(X)$ be the unique solution to
\begin{equation*}
 -\nabla\cdot\sigma_\eta\nabla u_\eta =0 \quad \mbox{ in } X,\qquad u_\eta=g \quad\mbox{ on } \partial X.
\end{equation*}
By Proposition~\ref{prop:asymptotic-2} we have $\norm{u_\eta - u^Z_{\rho_0}}_{H^{1-\delta}(Z)}\to 0$ as $\eta\to 0$ for some $\delta\in (0,\frac{1}{2})$.  Arguing as in Step 3, by \eqref{eq:step2b} we obtain
\begin{equation}\label{eq:step3}
\nu\cdot(R\nabla u_{\eta_0})\ge 2\mu\quad\text{on $\partial B(0,r)$}
\end{equation}
for some $\eta_0>0$.

\medskip

\emph{Step 5: The case of a smooth conductivity.} 
Let $\sigma_{\eta_0}^\epsilon\in C^\infty(X)$ be the standard mollified version of $\sigma_{\eta_0}$ for $\epsilon\in (0,1)$, namely $\sigma^\epsilon_{\eta_0}=\sigma_{\eta_0}*\varphi_\epsilon$, where
\[
\vp_\eps(x)=\eps^{-3}\vp(x/\eps),\qquad 
\vp(x)=
\begin{cases}
c \,e^{1/(|x|^2-1)}  & \text{if $|x|<1$,}\\
0 & \text{if $|x|\ge1$,}
\end{cases}
\]
and $c$ is chosen in such a way that $\int_{\mathbb R^3}\vp(x)\,dx=1$.  It is well known that $\sigma_{\eta_0}^\epsilon\to\sigma_{\eta_0}$ in $L^2(X)$. Let $u^\epsilon\in H^1(X)$ be the unique solution to
\begin{equation*}
 -\nabla\cdot\sigma_{\eta_0}^\epsilon\nabla u^\epsilon =0 \quad \mbox{ in } X,\qquad u^\epsilon=g \quad\mbox{ on } \partial X.
\end{equation*}

Observe now that $v^\epsilon = u^\epsilon - u_{\eta_0}\in H^1(X)$ is the unique weak solution of
\begin{equation}\label{eq:veps}
 -\nabla\cdot\sigma_{\eta_0}\nabla v^\epsilon =\nabla\cdot((\sigma_{\eta_0}^\epsilon-\sigma_{\eta_0})\nabla u^\epsilon) \quad \mbox{ in } X,\qquad v^\epsilon=0 \quad\mbox{ on } \partial X.
\end{equation}
It is easy to see that $v^\eps\to 0$ in $L^2(X)$\footnote{Since $\sigma^\eps_{\eta_0}$ is uniformly bounded by below and above by positive constants independent of $\eps$, we have that $u^\eps$ is uniformly bounded in $H^1(X)$. In particular, $v^\eps$ is uniformly bounded in $H^1_0(X)$. Therefore, there exists $v\in H^1_0(X)$ such that $v^\eps\rightharpoonup v$ in $H^1_0(X)$, up to a subsequence. Thus, by the Rellich--Kondrachov theorem we have that $v^\eps\to v $ in $L^2(X)$. It remains to show that $v=0$. Testing \eqref{eq:veps} against any $w\in C^\infty(X)$ with compact support contained in $X$ we have
\[
\int_X \sigma_{\eta_0}\nabla v^\eps\cdot\nabla w\,dx = \int_X (\sigma_{\eta_0}^\epsilon-\sigma_{\eta_0})\nabla u^\epsilon\cdot\nabla w\,dx.
\]
Since $\nabla v^\eps\rightharpoonup \nabla v$ in $L^2(X)$, the left hand side of this equality converges to $\int_X \sigma_{\eta_0}\nabla v\cdot\nabla w\,dx$ as $\eps\to 0$. On the other hand, we have
\[
|\int (\sigma_{\eta_0}^\epsilon-\sigma_{\eta_0})\nabla u^\epsilon\cdot\nabla w\,dx|\le 
\norm{\sigma_{\eta_0}^\epsilon-\sigma_{\eta_0}}_{L^2(X)}\norm{\nabla u^\epsilon}_{L^2(X)}\norm{\nabla w}_{L^\infty(X)}\underset{\eps\to 0}{\longrightarrow} 0.
\]
As a consequence, we have that $\nabla\cdot\sigma_{\eta_0}\nabla v=0$ in $X$, so that $v=0$.}.
Since $\sigma_{\eta_0}$ is constant in $Z$, for $\eps$ small enough we have that $\sigma_{\eta_0}^\epsilon-\sigma_{\eta_0}\equiv 0$ in $B(0,r)$. Thus, applying standard Schauder estimates (see \cite[Corollary 8.36]{GT-2001}) to \eqref{eq:veps} in $B(0,r)$ we obtain 
$
\norm{ v^\eps}_{C^1(\overline{B(0,r)})} \le C \norm{v^\eps}_{L^2(X)}
$, which implies
\[
\lim_{\epsilon\to 0} \norm{u^\epsilon - u_{\eta_0}}_{C^{1}\bigl(\overline{B(0,r)}\bigr)}=0.
\]
As a consequence, in view of \eqref{eq:step3} we can choose $\epsilon_0>0$ such that
\begin{equation*}
\nu\cdot(R\nabla u^{\epsilon_0})\ge \mu\quad\text{on $\partial B(0,r)$.}
\end{equation*}
Consider now the set of pathological conductivities given by
\[
P=\{\sigma\in C^\infty(\overline{X}):\text{$\sigma > 1/2$ in $X$, $\nu\cdot(R\nabla u^{\sigma})>0$ on $\partial B(0,r)$}\},
\]
where $u^\sigma\in H^1(X)$ is the unique solution to
\begin{equation*}
 -\nabla\cdot\sigma\nabla u^\sigma =0 \quad \mbox{ in } X,\qquad u^\sigma=g \quad\mbox{ on } \partial X.
\end{equation*}
We proved that $\sigma^{\epsilon_0}_{\eta_0}\in P$, so that $P\neq\emptyset $, and by construction $P$ is open.

\medskip

\emph{Step 6: The critical point.}
Finally, by the Brouwer fixed point theorem (see, e.g., \cite[Chapter 9.1]{evans}), for every $\sigma\in P$ the field $R\nabla u^{\sigma}$ must vanish somewhere in $B(0,r)$. Thus, $u^{\sigma}$ has a critical point in $B(0,r)$. This concludes the proof of the theorem.
\end{proof}

We generalize the preceding result to the case of a finite number of boundary conditions. For any finite number of boundary conditions, we can find a conductivity such that all the corresponding solutions have at least one critical point in $X$. In other words, considering multiple boundary conditions does not guarantee the absence of critical points for any of the corresponding solutions. More precisely, we have the following result.

\begin{theorem}\label{thm:multiple}
Let $X\subset\Rm^3$ be a bounded Lipschitz domain. Take $g_1,\dots,g_L\in C(\partial X)\cap H^{1/2}(\partial X)$. Then there exists a nonempty open set of conductivities $\sigma\in C^\infty(\overline{X})$, $\sigma\ge 1/2$  such that for every $l=1,\dots,L$, the solution $u^{l}\in H^1(X)$ to
\begin{equation*}
 -\nabla\cdot\sigma\nabla u^{l}=0 \quad \mbox{ in } X,\qquad u^{l}=g_l \quad\mbox{ on } \partial X
\end{equation*}
has at least one critical point in $X$, namely $\nabla u^{l}(x^{l})=0$ for some $x^{l}\in X$ (depending on $\sigma$).
\end{theorem}
\begin{proof}
Without loss of generality, assume that $X$ is connected and that  $g_l$ is not constant for every $l$. Consider the set
\[
A=\{(x^{1}_{(1)},x^{1}_{(2)},\dots,x^{L}_{(1)},x^{L}_{(2)})\in(\partial X)^{2L}:g_l(x^{l}_{(1)})\neq g_l(x^{l}_{(2)}),\;l=1,\dots L\}.
\]
Note that $A$ is non-empty (since $g_l$ is not constant) and relatively open in $(\partial X)^{2L}$ (since $g_l$ is continuous). Thus, we can choose $(x^{l}_{(i)})_{i=1,2}^{l=1,\dots,L}\in A$ such that all the points considered are distinct, namely
\[
\#\{x^{l}_{(i)}:i=1,2,\;l=1,\dots,L\}=2L.
\]
Without loss of generality, assume that $g_l(x^{l}_{(i)})=i$ for every $l$. Since the points are all distinct and we are in three dimensions, we can construct $L$ smooth open tubes $T_1,\dots,T_L\subset \Rm^3$ such that:
\begin{itemize}
\item the tubes are pairwise disjoint, namely $T_{l}\cap T_{l'}=\emptyset$ if $l\neq l'$;
\item and $x^{l}_{(i)}\in T_l$ for every $l=1,\dots,L$ and $i=1,2$.
\end{itemize}
In other words, the tube $T_l$ connects the two points $x^{l}_{(1)}$ and $x^{l}_{(2)}$.

We now construct suitable inclusions for each $l=1,\dots,L$. For $\rho\in (0,\tilde\rho)$ let  $Z^{l}$ and $X^{1,l}_{\rho}$, $X^{2,l}_{\rho}$ be  as in the proof of Theorem~\ref{thm:main}, corresponding to the points $x^{l}_{(1)}$ and $x^{l}_{(2)}$, constructed in such a way that $X^{1,l}_{\rho},X^{2,l}_{\rho},Z^{l}\subset T_l$. More precisely, $Z^{l}$ is obtained by translating, rotating and scaling $Z$, namely $Z^{l}=a_l Z+z_l$, where $a_l>0$ and $z_l\in T_l$ is the center of $Z^{l}$. The subdomains $X^{1,l}_{\rho}$ and $X^{2,l}_{\rho}$ are obtained  via smooth deformations of $X^1_{\rho}$ and $X^2_{\rho}$, and connect the boundary of $Z^{l}$ to  $x^{l}_{(1)}$ and $x^{l}_{(2)}$. Set
\[
\tilde Z = \bigcup_{l=1}^L Z^{l},\qquad \tilde X^i_{\rho} = \bigcup_{l=1}^L X^{i,l}_{\rho}.
\]

The rest of the proof is very similar to that of Theorem~\ref{thm:main}, with $\tilde Z$ and $\tilde X^i_{\rho}$ taking the role of $Z$ and $X^i_\rho$, respectively. The details are omitted.
\end{proof}

Before considering the case of Neumann boundary conditions, we consider another  generalization of Theorem~\ref{thm:main}: it is possible to construct conductivities yielding an arbitrary finite number of critical points located in arbitrarily small balls given a priori.

\begin{theorem}\label{thm:main-local}
Let $X\subset\Rm^3$ be a bounded Lipschitz domain and let $B_1,\dots,B_M\subseteq X$ be pairwise disjoint open balls. Take $g\in C(\partial X)\cap H^{\frac12}(\partial X)$. Then there exists a nonempty open set of conductivities $\sigma\in C^\infty(\overline{X})$, $\sigma\ge 1/2$, such that the solution $u\in H^1(X)$ to
\begin{equation*}
 -\nabla\cdot\sigma\nabla u =0 \quad \mbox{ in } X,\qquad u=g \quad\mbox{ on } \partial X
\end{equation*}
has a critical point in $B_m$ for every $m=1,\dots,M$, namely $\nabla u(x_m)=0$ for some $x_m\in X$ (depending on $\sigma$).
\end{theorem}
\begin{proof}
This result follows applying the same argument used in the proof of Theorem~\ref{thm:main}, the only difference lies in the construction of the inclusions where the conductivity takes large values.

If $g$ is constant, the result is obvious. Otherwise, for $i=1,2$ take $x_{(i)}\in\partial X$ such that $g(x_{(i)})=i$. For every $m=1,\dots,M$, let $Z_m$ be obtained by scaling and translating $Z$ in such a way that $Z_m\subset B_m$.  The $2M$ lateral discs of the cylinders $Z_m$ are connected to $x_{(1)}$ with a connected Lipschitz subdomain $X^1_{\rho}$ satisfying the assumptions of Lemma~\ref{lem:family}. Similarly, the $M$ lateral surfaces of $Z_m$ are connected to $x_{(2)}$ with a connected Lipschitz subdomain $X^2_{\rho}$ satisfying the assumptions of Lemma~\ref{lem:family}. In particular, $\partial X^i_{\rho}\cap\partial X=B(x_{(i)},\rho)\cap\partial X$ for $i=1,2$. Moreover, we choose $X^i_{\rho}$ in such a way that $X^i_{\rho}$, with respect to the decomposition of the boundary given by $(\partial X^i_{\rho}\cap\partial X)^\circ$ and $\partial X^i_{\rho}\setminus\partial X$, is creased, according to Definition~\ref{def:creased}.

Proceeding as in the proofs of Theorems~\ref{thm:main}, we obtain that for $\rho$ and $\eta$ small enough, the corresponding solution will have at least one critical point in each $Z_m\subset B_m$. Further, the topology of the gradient field is preserved by suitable smooth deformations of the conductivity, and the result is proved.
\end{proof}

\subsection{Neumann Boundary Conditions}
\label{sec:mainN}

We conclude this section by a construction of critical points when the prescribed boundary conditions are of Neumann type.  We consider only the case of a single boundary condition and of a single critical point, although the result also extends to a finite number of boundary conditions and critical points, as in the setting of Dirichlet boundary conditions.

\begin{theorem}\label{thm:main-neumann}
Let $X\subset\Rm^3$ be a connected bounded Lipschitz domain. Take $g\in C(\partial X)$ such that $\int_{\partial X}g\,ds=0$. Then there exists a nonempty open set of conductivities $\sigma\in C^\infty(\overline{X})$, $\sigma\ge 1/2$ such that the solution $u\in H^1(X)/\Rm$ to
\begin{equation*}
 -\nabla\cdot\sigma\nabla u =0 \quad \mbox{ in } X,\qquad \sigma\partial_\nu u=g \quad\mbox{ on } \partial X
\end{equation*}
has a critical point in $X$, namely $\nabla u(x)=0$ for some $x\in X$ (depending on $\sigma$).
\end{theorem}
\begin{proof}
The proof follows the same structure of the proof of Theorem~\ref{thm:main}, and so only the most relevant differences will be pointed out. Without loss of generality, assume  that $g\not\equiv 0$.

The construction of the subdomains $\Xii$ and $Z$ is very similar to the one presented above, with the only difference lying in the contact surfaces $D^i = \partial \Xii \cap \partial X$. Making the surfaces $D^i$ very small is not necessary in this context. On the other hand, we observe from our results obtained in Proposition \ref{prop:asymptotic-neumann} and the estimates in \eqref{eq:neumann-convergence} that the only requirement we need to verify is
\begin{equation}\label{eq:nonzero}
 \int_{D^1} g\,ds+\int_\Gamma vg\,ds\neq 0,
\end{equation}
where $v$ is the unique solution to
\[
         \Delta v=0\quad\text{in $X_+$,} \qquad
       \partial_\nu v=0 \quad\text{on $\Gamma$,} \qquad
         v=1 \quad\text{on $N^1$,}\qquad
         v=0 \quad\text{on $N^2$,}
\]
$N^i = \partial \Xii\setminus\overline{D^i}$ and $X_+=X\setminus(\overline{\Xu\cup \Xd})$.
Since $0\le v\le 1$ in $X_+$ by the Hopf lemma, \eqref{eq:nonzero} can be satisfied choosing $D_1\subsetneq\{x\in\Omega\,:\,g(x)>0\}$ and $D_2=\{x\in\Omega\,:\,g(x)<0\}$, which imply  $g\ge0$ on $\Gamma$.

In view of \eqref{eq:nonzero}, with the notation of Proposition~\ref{prop:asymptotic-neumann}, we have $\beta_1\neq \beta_2$. Thus, by Proposition~\ref{prop:asymptotic-neumann} the limit solution $u^*$ as $\eta\to 0$ in the cylinder $Z$ will have two different constant boundary values on the two discs and on the lateral surface. The rest of the proof follows exactly as in the proof  of Theorem~\ref{thm:main}.
\end{proof}

\section{The Zaremba Problem}
\label{sec:zaremba}

The two handles $\Xii$ constructed in the previous section are two disjoint subdomains of $X$ whose boundaries are allowed to meet on a small set (of $2-$Haussdorf measure zero). Moreover, Dirichlet conditions are imposed on their part of the boundary that coincides with that of $X$, whereas Neumann conditions are imposed on the rest of their boundaries. The Laplace equation with such mixed boundary conditions is referred to as the Zaremba problem. Following \cite{2007-mitreas}, we present here the results we need in this paper.

We consider the following mixed boundary value problem for the Laplacian.
Let $\Omega\subseteq \Rm^3$ be a bounded and  Lipschitz domain, such that each connected component of $\Omega$ has a connected boundary. Let $D,N\subseteq \partial \Omega$ be disjoint, open, such that $D\neq\emptyset$, $\overline{D}\cap\overline{N}=\partial D =\partial N$ and $\overline D\cup\overline N=\partial\Omega$. We consider
\begin{equation}\label{eq:zaremba}
\left\{ \begin{array}{ll}
         \Delta u=0&\quad\text{in $\Omega$,}\\
         u=g&\quad\text{on $D$,}\\
         \partial_\nu u=f &\quad\text{on $N$,}
        \end{array}
        \right.
\end{equation}
and are interested in stability estimates of the form
\[
\norm{u}_{H^s(\Omega)}\le C(\norm{g}_{H^{s-\frac{1}{2}}(D)}+\norm{f}_{H^{s-\frac{3}{2}}(N)}),
\]
for $s\in [1-\delta,1+\delta]$. This problem was studied in \cite{2007-mitreas} in the case $N\neq \emptyset$ and previously in \cite{1995-jerison-kenig} in the case $N=\emptyset$, and we report here the main results of interest in this paper.

We assume $D$ and $N$ to be \emph{admissible patches} as in \cite{2007-mitreas}: this essentially means that the interface between $D$ and $N$ is Lipschitz continuous. For the sake of completeness, we now provide a precise definition. For each point $x=(x_1,x_2,x_3)$ in $\mathbb R^3$, we set $x'=(x_1,x_2)$.
\begin{definition}
Let $\Omega\subset\mathbb R^3$ be a bounded Lipschitz domain. An open set $\Sigma\subset\partial\Omega$ is called an {\em admissible patch} if for every $x_0\in\partial\Sigma$ there exists a new system of orthogonal axes such that $x_0$ is the origin and the following holds. There exists a cube
$Q=Q_1\times Q_2\times Q_3\subset \mathbb R\times\mathbb R\times\mathbb R$ centered at $0$ and two Lipschitz functions
\begin{eqnarray*}
&&\varphi\,:\,Q'=Q_1\times Q_2\longrightarrow Q_3,\qquad \varphi(0)=0\\
&&\psi\,:\,Q_2\longrightarrow Q_1,\qquad \psi(0)=0,
\end{eqnarray*}
satisfying
\begin{eqnarray*}
&&\Sigma\cap Q=\{(x',\varphi(x'))\,:\, x'\in Q'\textrm{ and }\psi(x_2)<x_1\},\\
&&(\partial\Omega\setminus\Sigma)\cap Q=\{(x',\varphi(x'))\,:\, x'\in Q'\textrm{ and }\psi(x_2)>x_1\},\\
&&\partial\Sigma\cap Q=\{(\psi(x_2),x_2,\varphi(\psi(x_2),x_2))\,:\, x_2\in Q_2\}.
\end{eqnarray*}
\end{definition}
We also assume that $\Omega$, with the decomposition of the boundary given by $D$ and $N$, is a {\em creased} domain. In essence, this means that $D$ and $N$ are separated by a Lipschitz interface and the angle between $D$ and $N$ is smaller than $\pi$.
\begin{definition}
\label{def:spcreased}
Let $\Omega$  be a Lipschitz domain in $\mathbb R^3$ and suppose that $D,N\subset\partial\Omega$ are two non-empty, disjoint admissible patches satisfying $\overline D\cap\overline N=\partial D=\partial N$ and $\overline D\cup\overline N=\partial\Omega$. The domain $\Omega$ is called {\em special creased} provided that the following conditions hold.
\begin{itemize}
\item[(i)] There exists a Lipschitz function $\phi\,:\,\mathbb R^{2}\to\mathbb R$ with the property that
$\Omega=\{(x',x_3)\in\mathbb R^3\,:\,x_3>\phi(x')\}$.
\item[(ii)] There exists a Lipschitz function $\psi\,:\,\mathbb R\to\mathbb R$ such that
$$N=\{(x_1,x_2,x_3)\in\mathbb R^3\,:\,x_1>\psi(x_2)\}\cap\partial\Omega$$
and
$$D=\{(x_1,x_2,x_3)\in\mathbb R^3\,:\,x_1<\psi(x_2)\}\cap\partial\Omega.$$
\item[(iii)] There exist $\delta_D,\delta_N\geq0$ with $\delta_D+\delta_N>0$ such that
$$\frac{\partial\phi}{\partial x_1}\geq\delta_N \textrm{ almost everywhere on } \{(x_1,x_2,x_3)\in\mathbb R^3\,:\,x_1>\psi(x_2)\}$$
and
$$\frac{\partial\phi}{\partial x_1}\leq-\delta_D \textrm{ almost everywhere on } \{(x_1,x_2,x_3)\in\mathbb R^3\,:\,x_1<\psi(x_2)\}.$$
\end{itemize}
\end{definition}
\begin{definition}\label{def:creased}
Let $\Omega$ be a bounded Lipschitz domain in $\mathbb R^3$ with connected boundary and suppose $D,N\subset \partial\Omega$ are two non-empty disjoint admissible patches satisfying $\overline D\cap\overline N=\partial D=\partial N$ and $\overline D\cup\overline N=\partial\Omega$. The domain $\Omega$ is called {\em creased} provided that the following conditions hold.
\begin{itemize}
\item[(i)] There exist $P_i\in\partial\Omega$, $i=1,\dots,M$ and $r>0$ such that $\partial\Omega\subset\cup_{i=1}^M B(P_i,r)$.
\item[(ii)] For each $i=1,\dots,M$ there exist a coordinate system $\{x_1,x_2,x_3\}$ in $\mathbb R^3$ with origin at $P_i$ and a Lipschitz function $\phi_i\,:\,\mathbb R^{2}\to\mathbb R$ such that the set $\Omega_i=\{(x',x_3)\in\mathbb R^3:x_3>\phi_i(x')\}$, with boundary decomposition $\partial\Omega_i=N_i\cup D_i$, is a special creased domain in the sense of Definition \ref{def:spcreased} and
    $$\begin{array}{l}\Omega\cap B(P_i,2r)=\Omega_i\cap B(P_i,2r),\\
    D\cap B(P_i,2r)=D_i\cap B(P_i,2r),\\
    N\cap B(P_i,2r)=N_i\cap B(P_i,2r).\end{array}$$
\end{itemize}
\end{definition}

We have the following result on traces. While the results in \cite{2007-mitreas} are expressed in terms of general Besov spaces, here we only need the simpler case of Sobolev spaces using the identification $B^{2,2}_s = H^s$ \cite[Exercise 6.5.2]{2014-grafakos}.
\begin{lemma}[\cite{1984-jonsson-wallin,2007-mitreas}]\label{lem:traces}
Let $\Omega\subseteq \Rm^3$ be a bounded Lipschitz domain with connected boundary, and $\Sigma\subseteq\partial\Omega$ be an admissible patch. Take $s\in (\frac{1}{2},\frac{3}{2})$.
\begin{enumerate}
\item The trace operator ${\rm Tr}\colon H^s(\Omega)\to H^{s-\frac{1}{2}}(\partial\Omega)$ is bounded.
\item There exists a bounded extension operator ${\rm Ext}_{\partial\Omega}\colon H^{s-\frac{1}{2}}(\partial\Omega) \to H^s(\Omega)$ such that ${\rm Tr}\circ{\rm Ext}_{\partial\Omega}={\rm Id}$.
\item There exists a bounded extension operator ${\rm Ext}_\Sigma \colon H^{s-\frac{1}{2}}(\Sigma) \to H^{s-\frac{1}{2}}(\partial\Omega)$ such that  ${\rm R}_\Sigma\circ{\rm Ext}_\Sigma={\rm Id}$, where ${\rm R}_\Sigma u=u_{|\Sigma}$.
\item The trace operator
\[
{\rm Tr}_\nu\colon \{u\in H^s(\Omega):\Delta u=0 \text{ in }\Omega\}\longrightarrow H^{s-\frac{3}{2}}(\Sigma),\qquad u\mapsto \partial_\nu u_{|\Sigma}
\]
 is bounded.
\end{enumerate}
\end{lemma}
The main well-posedness result for the Zaremba problem then reads as follows.
\begin{proposition}[\cite{1995-jerison-kenig,2007-mitreas}]\label{prop:zaremba}
Under the above assumptions, there exists $\delta\in (0,\frac{1}{2})$ depending only on $\Omega$, $D$ and $N$ such that for every $s\in [1-\delta,1+\delta]$, problem \eqref{eq:zaremba} is well-posed and for every $g\in H^{s-\frac{1}{2}}(D)$ and $f\in H^{s-\frac{3}{2}}(N)$, we have
\[
\norm{u}_{H^s(\Omega)}\le C(\norm{g}_{H^{s-\frac{1}{2}}(D)}+\norm{f}_{H^{s-\frac{3}{2}}(N)})
\]
for some $C>0$ independent of $f$ and $g$. When $N=\emptyset$, we may choose $\delta=\frac{1}{4}$.
\end{proposition}

We conclude this section with a technical lemma on the Sobolev regularity of functions separately defined on subsets.
\begin{lemma}\label{lem:jumps}
Let $\Omega\subseteq \Rm^3$ be a bounded  Lipschitz domain with connected boundary, and $\Sigma_1,\Sigma_2\subseteq\partial\Omega$ be two disjoint admissible patches (with possibly non-disjoint boundaries).  Take $s\in (0,\frac{1}{2})$ and $g_i\in H^s(\Sigma_i)$ for $i=1,2$. Set $\Sigma=(\overline{\Sigma_1}\cup \overline{\Sigma_2})^\circ$ and define $g$ on $\Sigma$ by
\[
g=\chi_{\Sigma^1} g_1 +\chi_{\Sigma^2} g_2,
\]
where $\chi_S$ denotes the characteristic function of the set $S$.
Then $g\in H^s(\Sigma)$ and
\[
\norm{g}_{H^s(\Sigma)}\le C(\norm{g_1}_{H^s(\Sigma_1)} + \norm{g_2}_{H^s(\Sigma_2)})
\]
for some $C>0$ depending only on $\Sigma_1$, $\Sigma_2$ and $s$.
\end{lemma}
\begin{proof}
Note that $g$ may be rewritten as
\[
g=\chi_{\Sigma_1} ({\rm Ext}_{\Sigma_1} g_1) +\chi_{\Sigma_2} ({\rm Ext}_{\Sigma_2} g_2).
\]
Thus, the result follows from part 3 of Lemma~\ref{lem:traces} and the well-known fact that the characteristic function of the half space $\Rm^2_+$ is a multiplier for the space $H^s(\Rm^2)$ if and only if $s<\frac{1}{2}$ \cite[Corollary 3.5.1]{mazya}.
\end{proof}

\section{The Conductivity Equation with High Contrast}
\label{sec:highc}
We now consider the high-contrast problem with constant high conductivity equal to $\eta^{-1}$ in the handles $\Xii$ and unit conductivity in the rest of $X$. We generalize the results of \cite{caloz2010uniform} to the case of two inclusions (handles) that touch the boundary and are allowed to touch each other on a set of zero two-dimensional measure. We study the Dirichlet case in section~\ref{sub:dir}  and the Neumann case  in section~\ref{sec:Nasymp}.

Let $X\subset\Rm^3$ be a bounded and Lipschitz domain with boundary $\partial X$. Let $\Xu,\Xd\subset X$ be two disjoint (possibly not connected) Lipschitz subdomains, and we set $D^i = (\partial \Xii\cap \partial X)^\circ$, $N^i = \partial \Xii\setminus\overline{D^i}$, $X_-=\Xu\cup \Xd$ and $X_+=X\setminus\overline{X_-}$. Assume that for $ i=1,2$
\begin{subequations}
\begin{align}
 & D^i\neq\emptyset, \label{eq:inclusion-boundary-2}\\
  & \mathcal{H}_2(\partial \Xu\cap \partial \Xd)=0, \label{eq:inclusion-boundary-3}\\
  & \text{$\Xii$, with boundary decomposition given by $D^i$ and $N^i$, is creased,}\label{eq:inclusion-boundary-4}\\
  & \text{each connected component of $\Xii$ and $X_+$ has a connected boundary,}
\end{align}
\end{subequations}
where $\mathcal{H}_2$ denotes the two-dimensional Haussdorf measure. In addition to the assumption that the inclusions actually touch the boundary, we are assuming that the intersection of their boundaries is of measure zero with respect to the boundary measure. (See Figure~\ref{fig:domain0} for an example, and Figure~\ref{fig:domain} for a more involved example where $\Xii=X^i_{\rho}$.) In essence, condition \eqref{eq:inclusion-boundary-4} means that the angle between $\partial \Xii$ and $\partial X$ is smaller than $\pi$. The unit normal $\nu$ is oriented outward $X$ on $\partial X$ and outward $\Xii$ on $\partial \Xii$, thereby pointing inward $X_+$ on $N^i$, as in Figure~\ref{fig:domain0}.

\begin{figure}
\begin{center}
\includegraphics[width=.831\textwidth]{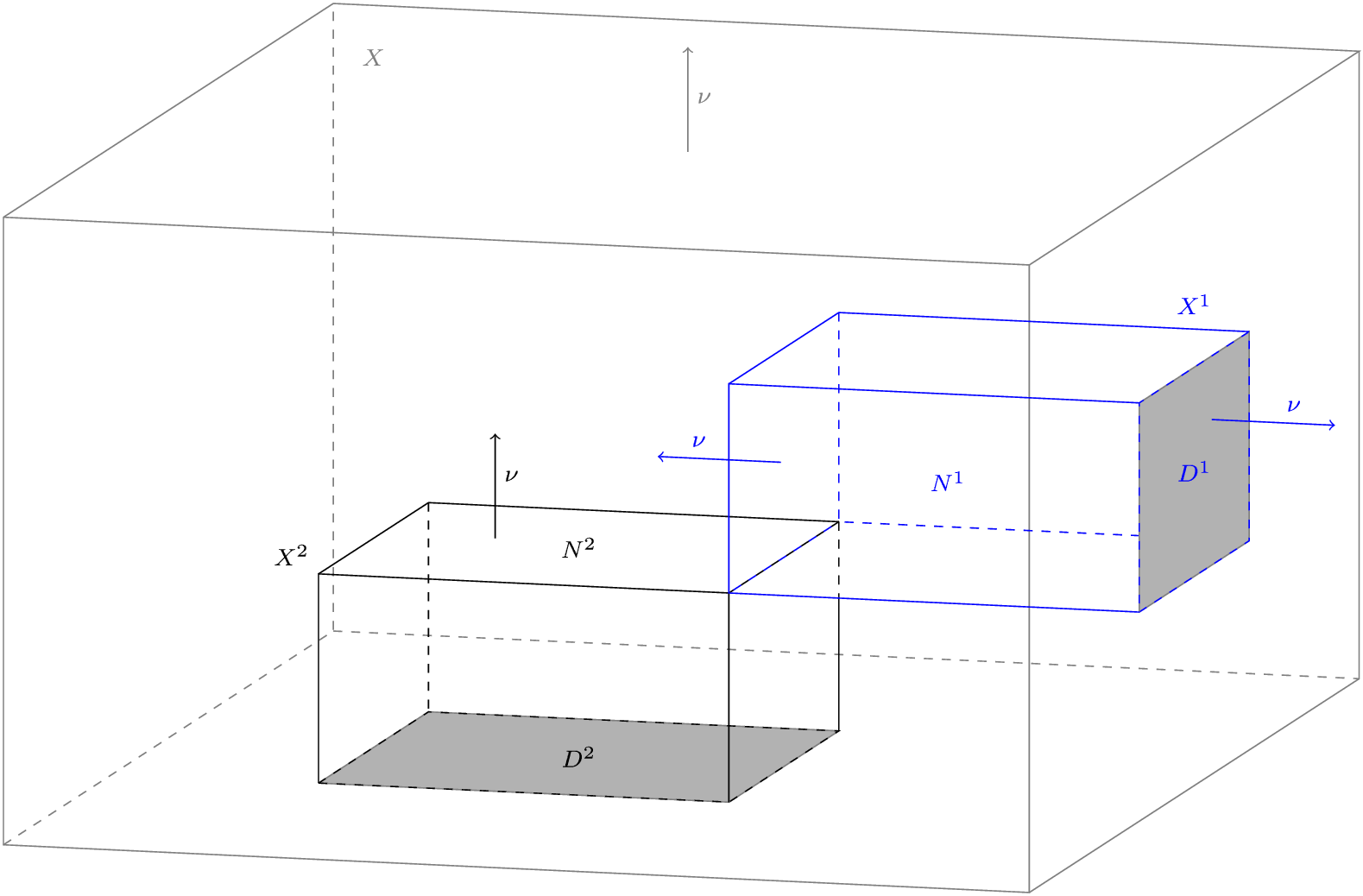}
\caption{A possible configuration of the domain $X$ and the inclusions $\Xu$ and $\Xd$. The shaded parts of the boundary represent $D_1$ and $D_2$, while the internal part of the boundary of the inclusions is formed by $N_1$ and $N_2$.\label{fig:domain0}}
\end{center}
\end{figure}

For $\eta\in(0,1)$, define the conductivity $\sigma_\eta\in L^\infty (X)$ by
\[
 \sigma_\eta=\begin{cases}
              1 & \text{in $X_+$,} \\
              \eta^{-1} & \text{in $X_-$.}
             \end{cases}
\]
\subsection{Dirichlet Boundary Conditions}\label{sub:dir}

For $g\in H^{\frac 12}(\partial X)$ let $u_\eta\in H^1(X)$ be the unique solution to
\begin{equation}\label{eq:u_eta-2}
 -\nabla\cdot\sigma_\eta\nabla u_\eta =0 \quad \mbox{ in } X,\qquad u_\eta=g \quad\mbox{ on } \partial X.
\end{equation}
We are interested in the limit of $u_\eta$ as $\eta\to 0$, i.e., as the conductivity of the inclusions tends to infinity. Let us denote the restriction of a function $\phi$ to $\Xii$ ($X_+$) by $\phi^{i}$ ($\phi^+$). Then we have:
\begin{proposition}\label{prop:asymptotic-2}
Under the above assumptions, there exist $C,\eta_0>0$ and $\delta\in (0,\frac{1}{2})$ depending only on $X$, $\Xu$ and $\Xd$ such that for every $\eta\in (0,\eta_0]$ there holds
\begin{align*}
& \left\Vert u^{i}_\eta - u_0^{i} \right\Vert_{H^{1-\delta}(\Xii)}\le C \left\Vert g\right\Vert_{H^{1/2}(\partial X)}\eta,\\
 & \left\Vert u^+_\eta - u_0^+ \right\Vert_{H^{1-\delta}(X_+)}\le C \left\Vert g\right\Vert_{H^{1/2}(\partial X)}\eta,
\end{align*}
where $u_0^{i}$ and $u_0^+$ are the unique solutions to the problems
\[
 \left\{ \begin{array}{ll}
         \Delta u_0^{i}=0&\quad\text{in $\Xii$,}\\
         u_0^{i}=g&\quad\text{on $D^i$,}\\
        \partial_\nu u_0^{i}=0 &\quad\text{on $N^i$,}
        \end{array}
        \right.
        \qquad
         \left\{ \begin{array}{ll}
         \Delta u_0^+=0&\quad\text{in $X_+$,}\\
         u_0^+=g&\quad\text{on $\partial X_+\cap \partial X$,}\\
         u_0^+=u_0^{i} &\quad\text{on $N^i$, $i=1,2$.}
        \end{array}
        \right.
\]
\end{proposition}
\begin{remark}
Note that we cannot take $\delta=0$, since for instance the boundary condition for $u^+_0$ has jumps, and so $u^+_0\notin H^1(X_+)$.
\end{remark}
\begin{remark}
In view of the Hopf lemma, the limiting solution in $\Xii$ satisfies
\[
 \inf_{D^i} g\le u_0^{i} \le  \sup_{D^i} g.
\]
This shows that the values of $u_0^{i}$ are controlled by the boundary conditions.
\end{remark}
We now prove Proposition~\ref{prop:asymptotic-2}, following the argument given in \cite{caloz2010uniform}, which we refer to for additional details.
\begin{proof}
For $i=1,2$, let $\delta^i\in(0,\frac{1}{2})$ be given by Proposition~\ref{prop:zaremba} for the set  $\Xii$ and the decomposition of the boundary given by $D^i$ and $N^i$ (cfr.\ Figure~\ref{fig:domain0}). Similarly, let $\delta^3\in(0,\frac{1}{2})$ be given by Proposition~\ref{prop:zaremba} for the set  $X_+$ and the decomposition of the boundary given by $\partial X_+$ and $\emptyset$ ($\delta^3=\frac{1}{4}$). Set $\delta=\min(\delta^1,\delta^2,\delta^3)$. For simplicity of notation, we denote $\Gamma=\partial X\setminus \partial X_-$.
Several different constants depending only on $\delta$, $X$, $\Xu$ and $\Xd$ will be denoted by $C$.

Problem \eqref{eq:u_eta-2} is equivalent to
\[
 \left\{ \begin{array}{ll}
         \Delta u_\eta^{i}=0&\quad\text{in $\Xii$,}\\
         u_\eta^{i}=g&\quad\text{on $D^i$,}\\
        \partial_\nu u_\eta^{i}=\eta \partial_\nu u_\eta^+ &\quad\text{on $N^i$,}
        \end{array}
        \right.
        \qquad
         \left\{ \begin{array}{ll}
         \Delta u_\eta^+=0&\quad\text{in $X_+$,}\\
         u_\eta^+=g&\quad\text{on $\Gamma$,}\\
         u_\eta^+=u_\eta^{i} &\quad\text{on $N^i$, $i=1,2$.}
        \end{array}
        \right.
\]
We look for solutions given by the asymptotic expansions
\begin{equation}\label{eq:asymptotic-2}
  u_\eta^+=\sum_{n=0}^\infty u^+_n \eta^n\quad \text{in $X_+$},\qquad
  u_\eta^{i}=\sum_{n=0}^\infty u^{i}_n \eta^n\quad \text{in $\Xii$.}
\end{equation}
The convergence of these series will be proved later. Inserting this ansatz into the above systems and identifying the same powers of $\eta$ we obtain
\[
         \left\{ \begin{array}{ll}
         \Delta u_n^+=0&\quad\text{in $X_+$,}\\
         u_n^+=\delta_0(n)\, g&\quad\text{on $\Gamma$,}\\
         u_n^+=u_n^{i} &\quad\text{on $N^i$, $i=1,2$.}
        \end{array}
        \right.
       \;
        \left\{ \begin{array}{ll}
         \Delta u_n^{i}=0&\quad\text{in $\Xii$,}\\
         u_n^{i}=\delta_0(n)\, g&\quad\text{on $D^i$,}\\
        \partial_\nu u_n^{i}=(1-\delta_0(n))\, \partial_\nu u_{n-1}^+ &\quad\text{on $N^i$,}
        \end{array}
        \right.
\]
with $\delta_0(0)=1$ and $\delta_0(n)=0$ for $n\geq1$.
Note that, by \eqref{eq:inclusion-boundary-3}, the boundary conditions set above follow from the identities $\partial \Xii=\overline{D^i}\cup \overline{N^i}$ and $\partial X_+ = \overline{\Gamma} \cup \overline{N^1}\cup\overline{N^2}$.

By Proposition~\ref{prop:zaremba} (applied to $X_+$ and the decomposition of the boundary given by $\partial X_+$ and $\emptyset$) and Lemma~\ref{lem:jumps} we have that the problem for $u^+_n$ is well-posed and that for $n\ge 0$ we have
\begin{equation*}\label{eq:u0+}
  \begin{split}
 \left\Vert u_n^+\right\Vert_{H^{1-\delta}(X_+)}&\le C \norm{\delta_0(n)\chi_{\Gamma}g+\chi_{N^1}u_n^{1}+\chi_{N^2}u_n^{2}}_{H^{1/2-\delta}(\partial X_+)}\\
  &\le C\left(\delta_0(n) \left\Vert  g \right\Vert_{H^{1/2}(\partial X)}+\sum_{i=1}^2 \left\Vert  u_{n}^{i}\right\Vert_{H^{1/2-\delta}(N^i)} \right).
  \end{split}
  \end{equation*}
 Thus, Lemma~\ref{lem:traces}, part 1, yields
  \begin{equation}\label{eq:first_system}
  \left\Vert u_n^+\right\Vert_{H^{1-\delta}(X_+)} \le C\left(\delta_0(n) \left\Vert  g \right\Vert_{H^{1/2}(\partial X)}+ \sum_{i=1}^2  \left\Vert u_{n}^{i}\right\Vert_{H^{1-\delta}(\Xii)}\right),\qquad n\ge 0.
  \end{equation}

  Similarly, by Proposition~\ref{prop:zaremba} (applied to $\Xii$ and the decomposition of the boundary given by $D^i$ and $N^i$) and Lemma~\ref{lem:traces}, part 4, we have
 \begin{align*}\label{eq:u0-}
 &  \left\Vert u_0^{i}\right\Vert_{H^{1-\delta}(\Xii)}\le C \left\Vert g\right\Vert_{H^{1/2-\delta}(D^i)}\le C  \left\Vert g \right\Vert_{H^{1/2}(\partial X)},\\
&  \left\Vert u_n^{i}\right\Vert_{H^{1-\delta}(\Xii)}\le C \left\Vert \partial_\nu u_{n-1}^+\right\Vert_{H^{-1/2-\delta}(N^i)}\le C  \left\Vert u_{n-1}^+\right\Vert_{H^{1-\delta}(X_+)},\qquad n\ge 1,\nonumber
\end{align*}
whence, by \eqref{eq:first_system}, we have
\begin{equation*}\label{eq:uones}
 \sum_{i=1}^2  \left\Vert u_{1}^{i}\right\Vert_{H^{1-\delta}(\Xii)} \le C \left\Vert u_{0}^+\right\Vert_{H^{1-\delta}(X_+)} \le C \left\Vert  g \right\Vert_{H^{1/2}(\partial X)}
\end{equation*}
and
\[
 \sum_{i=1}^2 \left\Vert u_{n}^{i}\right\Vert_{H^{1-\delta}(\Xii)}  \le C\sum_{i=1}^2  \left\Vert u_{n-1}^{i}\right\Vert_{H^{1-\delta}(\Xii)},\qquad n\ge 2.
\]
Thus, using again \eqref{eq:first_system} we obtain
\begin{equation}\label{eq:geometric}
\begin{aligned}
  &  \sum_{i=1}^2 \left\Vert u_{n+1}^{i}\right\Vert_{H^{1-\delta}(\Xii)}  \le C^n \sum_{i=1}^2  \left\Vert u_{1}^{i}\right\Vert_{H^{1-\delta}(\Xii)} \le C^{n+1} \left\Vert  g \right\Vert_{H^{1/2}(\partial X)},\\
  &\left\Vert u_{n+1}^+\right\Vert_{H^{1-\delta}(X_+)} \le C^{n+1} \sum_{i=1}^2  \left\Vert u_{1}^{i}\right\Vert_{H^{1-\delta}(\Xii)}  \le C^{n+1} \left\Vert  g \right\Vert_{H^{1/2}(\partial X)}
\end{aligned}
\end{equation}
for every $ n\ge 0$.

Define now $\eta_0=1/(2 C)$ and take $\eta\in (0,\eta_0]$. The above estimates show that $u^+_\eta$ and $u^{i}_\eta$ in \eqref{eq:asymptotic-2} are well defined.  By \eqref{eq:asymptotic-2} we can write $u_\eta^+-u_0^+=\eta \sum_{n=0}^\infty u^+_{n+1} \eta^{n}$ and $u_\eta^{i}-u_0^{i}=\eta \sum_{n=0}^\infty u^{i}_{n+1} \eta^{n}$. Hence, by \eqref{eq:geometric} we obtain
\begin{align*}
 &\left\Vert u_\eta^+-u_0^+\right\Vert_{H^{1-\delta}(X_+)}\le \eta \sum_{n=0}^\infty\eta^{n} \left\Vert u^+_{n+1}\right\Vert_{H^{1-\delta}(X_+)}\le C \eta \left\Vert  g \right\Vert_{H^{1/2}(\partial X)} \sum_{n=0}^\infty (C \eta)^n,\\
  &\left\Vert u_\eta^{i}-u_0^{i}\right\Vert_{H^{1-\delta}(\Xii)}\le \eta \sum_{n=0}^\infty\eta^{n} \left\Vert u^{i}_{n+1}\right\Vert_{H^{1-\delta}(\Xii)}\le C \eta \left\Vert  g \right\Vert_{H^{1/2}(\partial X)}  \sum_{n=0}^\infty (C \eta)^n.
\end{align*}
For $\eta\in (0,\eta_0]$, this implies
\begin{align*}
 &\left\Vert u_\eta^+-u_0^+\right\Vert_{H^{1-\delta}(X_+)}\le 2C \eta \left\Vert  g \right\Vert_{H^{1/2}(\partial X)},\\
 & \left\Vert u_\eta^{i}-u_0^{i}\right\Vert_{H^{1-\delta}(\Xii)}\le 2C \eta  \left\Vert  g \right\Vert_{H^{1/2}(\partial X)},
\end{align*}
as desired.
\end{proof}

\subsection{Neumann Boundary Conditions}
\label{sec:Nasymp}
We adapt here the results of the previous subsection to the case of Neumann boundary conditions. We make the same assumptions on $X$ and $\Xii$, and for simplicity we assume in addition that $X$ and $\Xii$ are connected for $i=1,2$. The conductivity $\sigma_\eta$ is defined as before, namely
\[
 \sigma_\eta=\begin{cases}
              1 & \text{in $X_+$,} \\
             \eta^{-1} & \text{in $X_-$.}
             \end{cases}
\]
Fix $x^1\in D^1$. For $g\in H^{-1/2}(\partial X)$ such that $\int_{\partial X} g\,ds = 0$, let $u_\eta\in H^1(X)$ be the unique solution to
\begin{equation}\label{eq:u_eta-neumann}
 -\nabla\cdot\sigma_\eta\nabla u_\eta =0 \quad \mbox{ in } X,\qquad \sigma_\eta\partial_\nu u_\eta=g \quad\mbox{ on } \partial X,\qquad u_\eta(x^1) =0.
\end{equation}
The last condition is set to enforce uniqueness. We are interested in the limit of $u_\eta$ as $\eta\to 0$, i.e.\ as the conductivity of the inclusions tends to infinity.
\begin{proposition}\label{prop:asymptotic-neumann}
Under the above assumptions, there exist $C,\eta_0>0$ and $\delta\in (0,\frac{1}{2})$ depending only on $X$, $\Xu$ and $\Xd$ such that for every $\eta\in (0,\eta_0]$ there holds
\begin{equation}\label{eq:neumann-convergence}
\begin{aligned}
& \left\Vert u^{i}_\eta - \beta_i \right\Vert_{H^{1-\delta}(\Xii)}\le C \left\Vert g\right\Vert_{H^{-1/2}(\partial X)}\eta,\\& \left\Vert u^+_\eta - u_0^+ \right\Vert_{H^{1-\delta}(X_+)}\le C \left\Vert g\right\Vert_{H^{-1/2}(\partial X)}\eta,
 \end{aligned}
\end{equation}
where
\[
\beta_1=0,\qquad \beta_2=-\left(\int_{N^2} \partial_\nu v\,ds\right)^{-1} \left( \int_{D^1} g\,ds+\int_\Gamma vg\,ds\right)
\]
and $v$ and $u_0^+$ are the unique solutions to the problems
\[
  \left\{ \begin{array}{ll}
         \Delta v=0&\quad\text{in $X_+$,}\\
       \partial_\nu v=0 &\quad\text{on $\partial X_+\cap \partial X$,}\\
         v=1 &\quad\text{on $N^1$,}\\
         v=0 &\quad\text{on $N^2$,}
        \end{array}
        \right.
        \qquad
         \left\{ \begin{array}{ll}
         \Delta u_0^+=0&\quad\text{in $X_+$,}\\
         \partial_\nu u_0^+=g&\quad\text{on $\partial X_+\cap \partial X$,}\\
         u_0^+=\beta_i &\quad\text{on $N^i$, $i=1,2$.}
        \end{array}
        \right.
\]
\end{proposition}
\begin{proof}
The proof is similar to that of Proposition~\ref{prop:asymptotic-2}, and so only a sketch will be provided. In particular, precise references to the well-posedness results are omitted.

Problem \eqref{eq:u_eta-neumann} is equivalent to
\[
 \left\{ \begin{array}{ll}
         \Delta u_\eta^{i}=0&\quad\text{in $\Xii$,}\\
         \partial_\nu u_\eta^{i}=\eta g&\quad\text{on $D^i$,}\\
        \partial_\nu u_\eta^{i}=\eta \partial_\nu u_\eta^+ &\quad\text{on $N^i$,}
        \end{array}
        \right.
        \qquad
         \left\{ \begin{array}{ll}
         \Delta u_\eta^+=0&\quad\text{in $X_+$,}\\
         \partial_\nu u_\eta^+=g&\quad\text{on $\Gamma$,}\\
         u_\eta^+=u_\eta^{i} &\quad\text{on $N^i$, $i=1,2$,}
        \end{array}
        \right.
\]
together with the condition $u^{1}_\eta(x^1)=0$.
We look for solutions given by the asymptotic expansions
\begin{equation}\label{eq:asymptotic-neumann}
  u_\eta^+=\sum_{n=0}^\infty u^+_n \eta^n\quad \text{in $X_+$},\qquad
  u_\eta^{i}=\sum_{n=0}^\infty u^{i}_n \eta^n\quad \text{in $\Xii$.}
\end{equation}
Inserting this ansatz into the above systems and identifying the same powers of $\eta$ yields
\[
        \left\{ \begin{array}{ll}
         \Delta u_n^{i}=0&\;\text{in $\Xii$,}\\
         \partial_\nu u_n^{i}=\delta_1(n)\, g&\;\text{on $D^i$,}\\
        \partial_\nu u_n^{i}=(1-\delta_0(n))\, \partial_\nu u_{n-1}^+ &\;\text{on $N^i$,}
        \end{array}
        \right.
        \;
         \left\{ \begin{array}{ll}
         \Delta u_n^+=0&\;\text{in $X_+$,}\\
        \partial_\nu u_n^+=\delta_0(n)\, g&\;\text{on $\Gamma$,}\\
         u_n^+=u_n^{i} &\;\text{on $N^i$, $i=1,2$,}
        \end{array}
        \right.
\]
together with $u^{1}_n(x^1)=0$. These problems should be solved in order following the sequence
\[
u_0^{i} \to u_0^+ \to u_1^{i} \to u_1^+ \to \dots \to  u_n^{i} \to u_n^+  \to u_{n+1}^{i} \to u_{n+1}^+\to\dots
\]
Note that, given $u_n^{i}$, the problem for $u_n^+$ is well-posed and admits a unique solution. Similarly, given $u^+_{n-1}$, the problem for $u_n^{1}$ is uniquely solvable because of the additional condition $u^{1}_n(x^1)=0$. On the other hand, $u_n^{2}$ is  determined up to a constant. In other words, we can write $u^{2}_n=\tilde u^{2}_n+a_n$, where $\tilde u^{2}_n$ is the solution to the problem such that $\tilde u^{2}_n(x^2)=0$ for a fixed $x^2\in D^2$ and $a_n\in\Rm$.  This constant is uniquely determined by imposing that the Neumann boundary conditions for $u_{n+1}^{i}$ have zero mean. (Note that this is automatically satisfied for  $u_{0}^{i}$.) More precisely,  we need to ensure that
\begin{equation}\label{eq:compatible}
\delta_0(n)\int_{D^1} g\,ds+\int_{N^1} \partial_\nu u^+_n\,ds=0.
\end{equation}
Since $g$ has zero mean on $\partial X$, it is enough to consider only this condition, which implies the corresponding identity for $i=2$. The Green's identity yields (note that the normal on $N^i$ is pointing inwards, yielding a sign change):
\[
\begin{split}
0 &= \int_{\partial X^+} u_n^+ \partial_\nu v- v \partial_\nu u_n^+\,ds\\
& =-\int_{N^1} u_n^{1}\partial_\nu v\,ds-\int_{N^2} \tilde u_n^{2}\partial_\nu v\,ds-a_n\alpha+\int_{N^1} \partial_\nu u^+_n\,ds-\delta_0(n)\int_\Gamma vg\,ds,
\end{split}
\]
where $\alpha=\int_{N^2} \partial_\nu v\,ds$.
Therefore, \eqref{eq:compatible} is equivalent to
\[
\delta_0(n)\int_{D^1} g\,ds+\delta_0(n)\int_\Gamma vg\,ds+\int_{N^1} u_n^{1}\partial_\nu v\,ds+\int_{N^2} \tilde u_n^{2}\partial_\nu v\,ds+a_n\alpha=0,
\]
which shows that $a_n$ is uniquely determined, since $\alpha\neq 0$ by the Hopf lemma. In particular, as $u^{1}_0\equiv  0$ and $\tilde u^{2}_0\equiv 0$, we have
\[
u^{2}_0\equiv a_0 = -\alpha^{-1} \left( \int_{D^1} g\,ds+\int_\Gamma vg\,ds\right).
\]

We have shown that all the above problems are well-posed and have unique solutions. Arguing as in the proof of Proposition~\ref{prop:asymptotic-2}, we prove \eqref{eq:neumann-convergence} and the convergence of the expansions given in \eqref{eq:asymptotic-neumann}.
\end{proof}

\appendix\normalsize
\section{The limit solution $u^*$}\label{sec:appendix}

For the sake of completeness (it is not required for the proofs), we  derive an explicit expression for $u^*$, the solution to \eqref{eq:ustar-pde}. The advantage of the cylindrical geometry is that $u^*$ may be expanded over an explicit basis of harmonic functions. Since the solution $u^*$ is piecewise constant on the boundary of the cylinder, its decomposition in that basis still involves an infinite number of terms. Using the symmetries of the geometry, we can analyze these terms and obtain quantitative information about the Hessian at the origin $O$.

We write the Laplacian in cylindrical coordinates $x_1=z$ and $(x_2,x_3)=(r\cos\phi,r\sin\phi)$, with $-\frac H2<z<\frac H2$, $0<r<a$, and $\phi\in[0,2\pi)$. In the text, we chose $H=4$ and $a=1$. Let $u$ equal $u^*-1$ so that it equals $0$ on the lateral disks of the boundary of the cylinder and $1$ elsewhere on the boundary.  Since the geometry is invariant by rotation, we obtain that $u=u(r,z)$ solves $r^{-1}\partial_r r\partial_r u+\partial^2_z u=0$ with $u(-\frac H2,r)=u(\frac H2,r)=0$ and $u(z,a)=1$.  

The function $u$ is symmetric in $z$ and so is harmonic in the cylinder with lateral boundary conditions $u(\pm\frac H2,r)=0$ and  $\partial_z u(0,r)=0$. Writing harmonic solutions with such boundary conditions as $u(z,r)=f(z)g(r)$, we find a (spectral) basis for such functions with basis elements $V_k(r,z)=\cos(\frac{(2k+1)\pi z}H) I_0(\frac{(2k+1)\pi r}H)$ for $k\geq0$. Here, $I_0(r)$ is the modified Bessel function of order $0$. As a consequence, we have the decomposition
\begin{equation}\label{eq:u}
  u(r,z) = \dsum_{k\geq0} u_k \cos(\frac{(2k+1)\pi z}H) I_0 (\frac{(2k+1)\pi r}H).
\end{equation}
We extend $u$ by oddness and by periodicity outside $(-\frac H2,\frac H2)$ so that we have a $2H$ periodic function  even about $0$ and odd about $\frac H2$. Let $h(z)$ be the above extension of the boundary condition $1$, i.e., $h(z)=1$ on $|z|<\frac H2$ and $h(z)=-1$ on $\frac H2<|z|<H$. Evaluating at $r=a$, where $u(z,a)$ equals $h(z)$, we get
\begin{displaymath}
  \int_{-H}^H h(z) \cos(\frac{(2k+1)\pi z}H)\,dz = u_{k} \int_{-H}^H \cos^2(\frac{(2k+1)\pi z} H) \,dz\, I_0(\frac{(2k+1)\pi a}H),
\end{displaymath}
which yields
\begin{displaymath}
   u_k =  \frac{4(-1)^k}{\pi (2k+1) I_0(\frac{(2k+1)\pi a}H)}.
\end{displaymath}


Since $u$ is real-analytic away from the boundary of the domain, we can differentiate \eqref{eq:u} at will. By symmetry, we deduce that $\partial_r u(0,0)=\partial_z u(0,0)=\partial_{rz} u(0,0)=0$. Since $I_0(0)=1$, we also obtain that
\begin{displaymath}
 \pdrr uz (0,z) = -\dsum_{k\geq0} u_{k} \big(\dfrac{(2k+1)\pi}{H}\big)^2  \cos (\frac{(2k+1)\pi z}H).
\end{displaymath}
Evaluating at $z=0$, we obtain 
\[
\frac{\partial^2 u^*}{\partial x_1^2} (O)=\pdrr uz (0,0)= -\frac{4\pi}{H^2}\dsum_{k\geq0} \frac{(-1)^k (2k+1)}{ I_0((2k+1)\pi aH^{-1})}   .
\]
It can be verified that the series on the right hand side is always negative, namely $\partial_{x_1}^2 u^*(O)=-2\lambda$ for some $\lambda>0$. In the other direction, we compute
\begin{displaymath}
   \pdrr ur (r,0) = \dsum_{k\geq0} u_{k} \big(\dfrac{(2k+1)\pi}{H}\big)^2  I''_0 (\frac{(2k+1)\pi r}H),
\end{displaymath}
and, since $I''_0(0) = \frac{1}{2}$, we obtain that $\partial^2_{x_{2,3}} u^*(O)=\partial^2_r u (0,0)=\lambda>0$, as expected.
Figure~\ref{fig:lambda} allows to understand the dependence of $\lambda$ on the geometry of the cylinder.

\begin{figure}

\begin{centering}
\includegraphics[width=.92\textwidth]{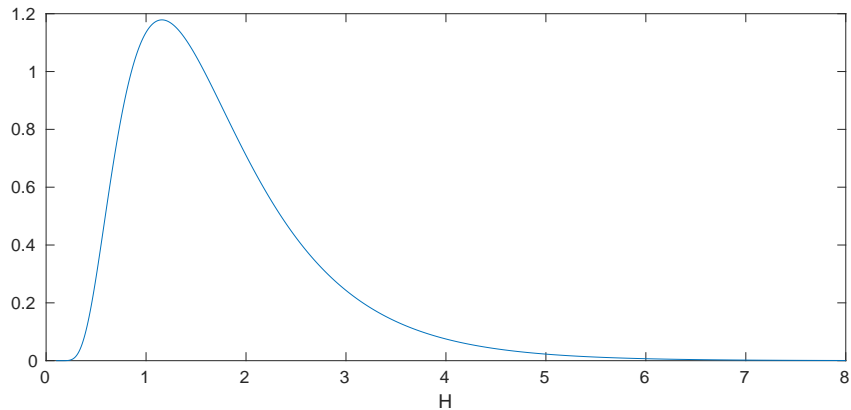}
\caption{\label{fig:lambda}The quantity $\lambda$ as a function of $H$, the height of the cylinder, with $a=1$.}
\end{centering}
\end{figure}

\bibliography{bibliography}
\bibliographystyle{abbrv}

\end{document}